\documentclass[a4paper,11pt]{article}
\usepackage[utf8]{inputenc}
\usepackage[T1]{fontenc}
\usepackage[english]{babel}
\usepackage[margin=2.8cm]{geometry}
\usepackage{mathpazo}
\usepackage[scaled=.95]{helvet}
\usepackage{courier}
\usepackage{amsmath,amsfonts,amssymb,amsthm,mathtools}
\usepackage{graphicx}
\usepackage{float}
\usepackage{natbib}
\usepackage{color}
\usepackage[onehalfspacing]{setspace}

\usepackage{slashbox}

\newtheorem{defi}{Definition}[section] 
\newtheorem{rmq}{Remark}[section]  
\newtheorem{theo}{Theorem}[section]  
\newtheorem{lem}{Lemma}[section]

\newtheorem{cor}{Corollary}[section]

\newcommand{\limite}[2]{\xrightarrow[#1]{#2}}

\DeclareUnicodeCharacter{B0}{\textdegree}
\newcommand{\Rlogo}{\protect\includegraphics[height=1.8ex,keepaspectratio]{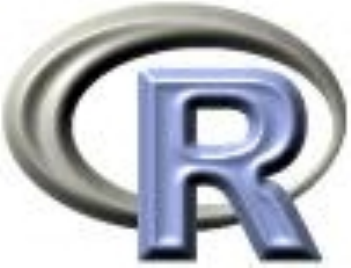}}

\def\bkN{\mathbb{N}}

\def\pa#1{\left( #1 \right)}

\newcommand\1{\leavevmode\hbox{\rm \small1\kern-0.35em\normalsize1}}
\newcommand\ind[1]{\1_{\{#1\}}}

\def\pa#1{\left( #1 \right)}

\def\build#1_#2^#3{\mathrel{
\mathop{\kern 0pt#1}\limits_{#2}^{#3}}}
\def\tend_#1^#2{\mathrel{
\mathop{\kern 0pt\longrightarrow}\limits_{#1}^{#2}}}

\def\bkR{\mathbb{R}}

\usepackage[draft]{fixme}
\fxusetheme{color}
\FXRegisterAuthor{h}{ah}{H}
\FXRegisterAuthor{p}{ap}{P}
\FXRegisterAuthor{a}{aa}{A}

\usepackage{hyperref}
\hypersetup{
pdfpagemode=UseOutlines,      
pdfstartview=Fit,             
pdffitwindow=true,            
pdfpagelayout=TwoColumnsRight,
pdftoolbar=true,              
pdfmenubar=true,              
bookmarksopen=false,          
bookmarksnumbered=true,       
colorlinks=true,              
pdfauthor={Ton nom},     
pdftitle={Titre PDF},    
pdfcreator=PDFLaTeX,          %
pdfproducer=PDFLaTeX,         %
linkcolor=blue,               
urlcolor=blue,                
anchorcolor=blue,            
citecolor=blue,              
frenchlinks=true,             
pdfborder={0 0 0}             
}

\usepackage[draft]{fixme}
\fxusetheme{color}
\FXRegisterAuthor{an}{aan}{A}
\FXRegisterAuthor{pe}{ape}{P}

\mathtoolsset{showonlyrefs,showmanualtags}

\begin{document}
\title{An efficient  Averaged Stochastic Gauss-Newton algorithm for estimating parameters of non linear regressions models}

\author{Peggy \textsc{C\'enac}$^{(*)}$, Antoine \textsc{Godichon-Baggioni}$^{(**)}$, Bruno \textsc{Portier}$^{(***)}$ \\ $^{(*)}$Institut de Math\'ematiques de Bourgogne, Universit\'e de Bourgogne, \\
$^{(**)}$Laboratoire de Probabilités, Statistique et Modélisation, Sorbonne Université \\
$^{(***)}$ Laboratoire de Mathématiques de l'INSA, INSA Rouen-Normandie 
} 
\date{}
\maketitle
\begin{abstract}
Non linear regression models are a standard tool for modeling real phenomena, with several applications in machine learning, ecology, econometry... Estimating the parameters of the model has garnered a lot of attention during many years. We focus here on a recursive method for estimating parameters of non linear regressions. Indeed, these kinds of methods, whose most famous are probably the stochastic gradient algorithm and its averaged version, enable to deal efficiently with massive data arriving sequentially. Nevertheless, they can be, in practice, very sensitive to the case where the eigenvalues of the Hessian of the functional we would like to minimize are at different scales. To avoid this problem, we first introduce an online  Stochastic Gauss-Newton algorithm. In order to improve the estimates behavior in case of bad initialization, we also introduce a new Averaged Stochastic Gauss-Newton algorithm and prove its asymptotic efficiency.  
\end{abstract}

\section{Introduction}
We consider in this paper  nonlinear regression model of the form
\begin{equation}
\label{model}Y_n = f(X_n,\theta) + \varepsilon_n, \ \ n\in\bkN,
\end{equation}
where the observations $(X_n,Y_n)_n$ are independent random vectors in $\mathbb{R}^p\times \mathbb{R}$, $(\varepsilon_n)_n$ are independent, identically distributed zero-mean non observable random variables. Moreover, $f: \mathbb{R}^p\times \mathbb{R}^q \longrightarrow \mathbb{R}$ 
and $\theta \in \mathbb{R}^q$ is the unknown 
parameter to estimate.
\vspace{1ex}

Nonlinear regression models are a standard tool for modeling real phenomena 
with a complex dynamic. 
It has a wide range of applications including maching learning (\cite{bach2014adaptivity}), ecology (\cite{komori2016asymmetric}), econometry (\cite{varian2014big}), pharmacokinetics (\cite{Bates88}), epidemiology (\cite{suarez2017}) or biology (\cite{Giurcuaneanu05}).
Most of the time, the parameter  $\theta$ is estimated using
the least squares method, thus it is estimated by
\[ \widehat\theta_n = \arg\min_{h \in \bkR^q} \sum_{j=1}^n (Y_j - f(X_j, h))^2 \]
Many authors have studied the asymptotic behavior of the least squares estimator,
under various assumptions with different methods. 
For example, \cite{Jennrich69} consider the case when $\theta$ belongs to a compact set, 
$f$ is a continuous function with a unique extremum. 
\cite{Wu81} consider similar local assumptions. 
\cite{lai94, Skouras, Geer90, Geer96, Yao00} generalized the consistency results. 
Under a stronger assumption about the errors $\varepsilon$, 
that is, considering that the errors are uniformly subgaussian,
\cite{Geer90} obtained sharper stochastic bounds using empirical process methods. Under second moment assumptions on the errors and regularity assumptions (lipschitz conditions) on $f$, \cite{Pollard06} established weak consistency and a central
limit theorem. 
\cite{Yang} consider a compact set for $\theta$, 
under strong regularity assumptions and the algorithm allows to reach a stationary point without certainty that it is the one we are looking for.
They built hypothesis tests and confidence intervals.
Finally, let also mention \cite{Yao00}  
which considers the case of stable nonlinear autoregressive models
and establishes the strong consistency of the least squares estimator.
\vspace{1ex}

However, in practice, the calculation of the least square estimator is not explicit in most cases and therefore requires the implementation of a deterministic approximation algorithm. The second order Gauss-Newton algorithm is generally used  (or sometimes Gauss-Marquard algorithm to avoid inversion problems 
(see for example \cite{Bates88}).

 These algorithms are  therefore not adapted to the case where the data are acquired sequentially and at high frequencies. 
 In such a situation, stochastic algorithms
offer an interesting alternative. 
One example is the stochastic gradient algorithm,
defined by 
\begin{equation}
\theta_n = \theta_{n-1} - \gamma_n \nabla_\theta f(X_n, \theta_{n-1}) \pa{Y_n - f(X_n, \theta_{n-1})},
\end{equation}
where  $(\gamma_n)$ is a sequence of  positive real numbers decreasing towards 0. Thanks to its recursive nature, 
this algorithm does not require to store all the data and can be updated automatically when the data sequentially arrive. 

It is thus adapted to very large datasets of data arriving at high speed. 
We refer to \cite{robbins1951} and to its averaged version \cite{PolyakJud92} 
and \cite{ruppert1988efficient}. 
However, even if it is often used in practice as in the case of neural networks,
the use of stochastic gradient algorithms can lead to non satistying results. Indeed, for such models, it amounts to take the same step sequence for each coordinates. Nevertheless, as explained in \cite{BGBP2019}, if the Hessian of the function we would like to minimize has eigenvalues at different scales, 
this kind of "uniform" step sequences is not adapted.
\vspace{1ex}

In this paper, we propose an alternative strategy to stochastic gradient algorithms, 
in the spirit of the Gauss-Newton algorithm.
It is defined by :
\begin{eqnarray}
\phi_{n} & = & \nabla_\theta f(X_{n}, \theta_n)
\label{GNun} \\
S_{n}^{-1} &= & S_{n-1}^{-1} - (1 + \phi_{n}^T  S_n^{-1} \phi_{n})^{-1}
S_{n-1}^{-1} \phi_{n} \phi_{n}^T S_{n-1}^{-1} 
\label{GNdeux} \\
\theta_{n} &= &\theta_{n-1} + S_{n}^{-1}\phi_{n}\pa{Y_{n} -f(X_{n}, \theta_{n-1})}
\label{GNtrois}
\end{eqnarray}
where the initial value $\theta_0$ can be arbitrarily chosen and $S_0$ is a positive definite deterministic matrix,
typically $S_0 = I_q$ where $I_q$ denotes the identity matrix of order $q$.
Remark that thanks to Riccati's formula (see \cite[p. 96]{Duf97}, also called Sherman Morrison's formula (\cite{Sherman50}), $S_n^{-1}$ is the inverse of the matrix $S_{n}$ defined by
\[S_n = S_0 + \sum_{j=1}^{n}\phi_j \phi_j^T.\]

When the function $f$ is linear of the form $f(x,\theta) = \theta^T x$,
 algorithm (\ref{GNun})-(\ref{GNtrois})  rewrites as the standard recursive least-squares estimators
(see \cite{Duf97}) defined by:
\begin{eqnarray*}
S_n^{-1} &= & S_{n-1}^{-1} - (1 + X_n^T  S_{n-1}^{-1} X_n)^{-1}
S_{n-1}^{-1} X_nX_n^T S_{n-1}^{-1} \\
\theta_n &= &\theta_{n-1} + S_{n}^{-1}X_n\pa{Y_n - \theta_{n-1}^T X_n}  
\end{eqnarray*}
This algorithm can be considered as a Newton stochastic algorithm since the matrix $n^{-1} S_n$ is an estimate of the Hessian matrix of the least squares criterion.
\vspace{1ex}

To the best of our knowledge and apart from the least squares estimate mentioned above, 
second order stochastic algorithms are hardly ever used and studied since they often require 
the inversion of a matrix at each step, which can be very expensive in term of time calculation. 
To overcome this problem some authors (see for instance \cite{mokhtari2014res,lucchi2015variance,byrd2016stochastic})
use the BFGS (for \emph{Broyden-Fletcher-Goldfarb-Shanno}) algorithm
which is based on the recursive estimation of a matrix 
whose behavior is closed to the one of the inverse of the Hessian matrix. 
Nevertheless, this last estimate need a regularization of  the objective function,
leading to unsatisfactory estimation of the unknown parameter.
In a recent paper  dedicated to estimation of parameters in logistic regression models  \citep{BGBP2019}, 
the authors propose a truncated Stochastic Newton algorithm. This truncation opens the way for online stochastic Newton algorithm without necessity to penalize the objective function. 

In the same spirit of this work, and to relax assumption on the function $f$,
we  consider in fact a modified version  of the
 stochastic Gauss-Newton algorithm defined by \eqref{GNtrois}, 
that enables us to obtain the asymptotic efficiency of the estimates
in a larger  area of assumptions. 
In addition, we introduce the following new \emph{Averaged Stochastic Gauss-Newton} algorithm (ASN for short)
defined by
\begin{eqnarray}
\phi_n & = & \nabla_h f(X_n, \overline{\theta}_{n-1})
\label{AGNun} \\
S_n^{-1} &= & S_{n-1}^{-1} - (1 + \phi_n^T  S_{n-1}^{-1} \phi_n)^{-1}
S_{n-1}^{-1} \phi_n \phi_n^T S_{n-1}^{-1} 
\label{AGNdeux} \\
\theta_n &= &\theta_{n-1} + n^\beta S_n^{-1}\phi_n\pa{Y_n -f(X_n, \theta_{n-1})}
\label{AGNtrois} \\
\overline{\theta}_n &=& \overline{\theta}_{n-1}
\label{AGNquatre}  + \dfrac{1}{n}\,\pa{\theta_n - \overline{\theta}_{n-1}}
\end{eqnarray}
where $\beta\in (0, 1/2)$, $\overline{\theta}_0 = 0$.

The introduction of the term $n^\beta$ before the term $S_n^{-1}$  in (\ref{AGNtrois}) allows the algorithm to move quickly which enables to reduce the sensibility to a bad initialization.  
The averaging step allows to maintain an optimal asymptotic behavior. 
Indeed, under assumptions, we first give the rate of convergence of the estimates, before proving their asymptotic efficiency.
\vspace{1ex}

The paper is organized as follows.
Framework and algorithms are introduced in Section~\ref{sec:framework}.
In Section~\ref{sectiontheo}, we give the almost sure rates of convergence of the estimates and establish their asymptotic normality. A simulation study illustrating the interest of averaging is presented in Section~\ref{sectionsimu}. Proofs are postponed in Section~\ref{sectionproof} while some general results used in the proofs on almost sure rates of convergence for martingales  are given in Section~\ref{sectionmartingales}.


\section{Framework}
\label{sec:framework}
\subsection{The model}
Let us consider the non linear regression model of the form 
\[
Y_{n} = f \left( X_{n} , \theta \right) + \epsilon_{n}
\]
where $( X_n, Y_n,\epsilon_n)_{n\geq 1}$ is a sequence of independent and identically distributed random vectors in $\mathbb{R}^{p} \times \mathbb{R} \times \mathbb{R}$. 
Furthermore, for all $n$, $\epsilon_n$ is independent from $X_n$ 
and is a zero-mean random variable. 
In addition, the function $f$ is assumed to be almost surely twice  differentiable with respect to the second variable.
Under certain assumptions, $\theta$ is a local minimizer of the functional $G: \mathbb{R}^{q} \longrightarrow \mathbb{R}_{+}$ defined for all $h \in \mathbb{R}^{q}$ by
\[
G(h) = \frac{1}{2}\mathbb{E}\left[ \left( Y - f \left( X,h \right) \right)^{2} \right] =: \frac{1}{2} \mathbb{E}\left[ g(X,Y,h) \right] ,
\]
where $(X, Y, \epsilon)$ has the same distribution as $(X_1, Y_1, \epsilon_1)$. 
Suppose from now that the following assumptions are fulfilled:
\begin{itemize}
	\item[\textbf{(H1a)}] There is a positive constant $C$ such that for all $h \in \mathbb{R}^{q}$, 
	\[
	\mathbb{E}\left[ \left\| \nabla_h g\left( X,Y,h \right) \right\|^{2}\right] \leq C,
	\]
	\item[\textbf{(H1b)}] There is a positive constant $C''$ such that for all $h \in \mathbb{R}^{q}$,
	\[
	\mathbb{E}\left[ \left\| \nabla_{h} f \left( X, h \right) \right\|^{4} \right] \leq C''
	\]
\item[\textbf{(H1c)}] The matrix $L(h)$ defined for all $h \in \mathbb{R}^{q}$ by
\[
L(h) = \mathbb{E}\left[ \nabla_{h}f (X,h)\nabla_{h}f(X,h)^{T} \right] 
\]
is positive at $\theta$.
\end{itemize}

Assumption \textbf{(H1a)} ensures ifrst that the functional $G$ is Frechet differentiable for all $h \in \mathbb{R}^{q}$. Moreover, since $\epsilon$ is independent from $X$ and zero-mean,
\[
\nabla G (h) = \mathbb{E} \left[ \left( f(X,h) - f(X,\theta) \right) \nabla_{h} f \left( X,h \right) \right] .
\]
Then, $\nabla G(\theta) = 0$. 
Assumption \textbf{(H1b)} will be crucial to control the possible divergence of estimates of $L(\theta)$ as well as to give their rate of convergence. Finally, remark that thanks to assumption \textbf{(H1c)},  $L(\theta)$ is invertible.

\subsection{Construction of the Stochastic Gauss-Newton algorithm}
In order to estimate $\theta$, we propose in this work a new approach. 
Instead of using straight a Stochastic Newton algorithm based on the estimate of the Hessian of the functionnal $G$, we will substitute this estimate by an estimate of $L(\theta)$, imitating thereby the Gauss-Newton algorithm. This leads to an algorithm of the form
\begin{equation}
\label{SGNalgo}
\theta_{n+1} =\theta_{n} + \frac{1}{n+1}L_n^{-1}\left(Y_{n+1} - f\left( X_{n+1},\theta_{n}\right)\right)\nabla_{h}f\left( X_{n+1} , \theta_{n} \right),
\end{equation}
where 
\[L_{n}:=\frac{1}{n}\sum_{i=1}^{n} \nabla_{h}f \left( X_{i} , \theta_{i-1} \right) \nabla_{h}f \left( X_{i} , \theta_{i-1} \right)^{T} \]
is a natural recursive estimate of $L(\theta)$. Remark that supposing the functional $G$ is twice differentiable leads to 
\[
\nabla^{2}G(h) = \mathbb{E}\left[ \nabla_{h} f \left( X,h \right) \nabla_{h} f\left( X,h \right)^{T} \right] - \mathbb{E}\left[ \left( f(X,\theta) - f(X,h) \right) \nabla_{h}^{2}f(X,h) \right],
\]
and in particular, $ H = \nabla^{2}G(\theta) =L(\theta)$. Then, $L_{n}$ is also an estimate of $H$. The proposed algorithm can so be considered as a Stochastic Newton algorithm, but does not require an explicit formula for the Hessian of $G$. Furthermore, the interest of considering $L_{n}$  as an estimate of $H$ is that we can update recursively the inverse of $L_{n}$ thanks to the Riccati's formula. To obtain the convergence of such an algorithm, it should be possible to state the following assumption:
\begin{itemize}
	\item[\textbf{(H*)}] There is a positive constant $c$ such that for all $h \in \mathbb{R}^{q}$,
	\[
	\lambda_{\min} \left( L (h) \right) \geq c.
	\]
\end{itemize}
Nevertheless, this assumption consequently limits the family of functions $f$ we can consider.

\subsection{The algorithms}
In order to free ourselves from the restrictive previous assumption \textbf{(H*)}, 
we propose to estimate $\theta$ with the following Gauss-Newton algorithm.
\begin{defi}[Stochastic Gauss-Newton algorithm]
Let $(Z_n)_{n \geq 1}$ be a sequence of random vectors, independent and for all $n\geq 1$, $Z_n \sim \mathcal{N}(0,I_q)$. 
The Gauss-Newton algorithm is defined recursively for all $n \geq 0$ by
\begin{align}
\notag & \widetilde{\Phi}_{n+1} = \nabla_{h} f \left( X_{n+1},\widetilde{\theta}_{n} \right) \\
\label{defsgn}& \widetilde{\theta}_{n+1} = \widetilde{\theta}_{n} + \widetilde{H}_{n}^{-1}\widetilde{\Phi}_{n+1} \left( Y_{n+1} - f \left( X_{n+1},\widetilde{\theta}_{n} \right)\right) \\
\notag &  \widetilde{H}_{n+ \frac{1}{2}}^{-1} = \widetilde{H}_{n}^{-1} - \left( 1+ \frac{c_{\widetilde{\beta}}}{(n+1)^{\widetilde{\beta}}}Z_{n+1}^{T}\widetilde{H}_{n}^{-1}Z_{n+1} \right)^{-1} \frac{c_{\widetilde{\beta}}}{(n+1)^{\widetilde{\beta}}} \widetilde{H}_{n}^{-1} Z_{n+1}Z_{n+1}^{T}\widetilde{H}_{n}^{-1} \\
\notag & \widetilde{H}_{n+1}^{-1} = \widetilde{H}_{n+\frac{1}{2}}^{-1} - \left( 1+ \widetilde{\Phi}_{n+1}^{T} \widetilde{H}_{n+ \frac{1}{2}}^{-1} \hat{\Phi}_{n+1} \right)^{-1} \widetilde{H}_{n+ \frac{1}{2}}^{-1} \widetilde{\Phi}_{n+1}\widetilde{\Phi}_{n+1}^{T} \widetilde{H}_{n+ \frac{1}{2}}^{-1} ,
\end{align}
with $\widetilde{\theta}_{0}$ bounded, $\widetilde{H}_{0}^{-1}$ symmetric and positive, $c_{\widetilde{\beta}} \geq 0$ and $\widetilde{\beta} \in (0 , 1/2)$.
\end{defi}
Note that compared with the algorithm (\ref{SGNalgo}), 
matrix $(n+1)L_n$ has been replaced by matrix $\widetilde{H}_n$ defined by:
\[
\widetilde{H}_n = \widetilde{H}_0 + \sum_{i=1}^{n} \widetilde{\Phi}_i\widetilde{\Phi}_i^T 
+ \sum_{i=1}^{n}\frac{c_{\widetilde{\beta}}}{i^{\widetilde{\beta}}}Z_iZ_i^T.
\]
Matrix $\widetilde{H}_n^{-1}$ is iteratively computed thanks to
Riccati's inversion formula (see \cite[p.96]{Duf97}) which is applied twice: first  recursively inverse matrix $\widetilde{H}_{n+\frac{1}{2}} = \widetilde{H}_{n} +\frac{c_{\widetilde{\beta}}}{(n+1)^{\widetilde{\beta}}}Z_{n+1}Z_{n+1}^T$, 
then  matrix $\widetilde{H}_{n+1} = \widetilde{H}_{n+\frac{1}{2}} + \widetilde{\Phi}_{n+1} \widetilde{\Phi}_{n+1}^T$.
\vspace{1ex}
In fact, introducing this additional term enables to ensure, taking $c_{\beta} > 0$ that (see the proof of Theorem~\ref{consistency})
\[
\lambda_{\max} \left( \widetilde{H}_{n}^{-1} \right) = \mathcal{O}\left( n^{1-\widetilde{\beta}} \right) \quad \mbox{a.s.}
\]
Therefore, it enables to control the possible divergence of the estimates of the inverse of the Hessian and to obtain convergence results without assuming  \textbf{(H*)}.
Anyway, if assumption \textbf{(H*)} is verified, one can take $c_{\widetilde{\beta}} = 0$ and Theorem \ref{theothetatilde} remains true. Remark that considering the Gauss-Newton algorithm amounts to take a step sequence of the form $\frac{1}{n}$. Nevertheless, in the case of stochastic gradient descents, it is well known in practice that this can lead to non sufficient results with a bad initialization. In order to overcome this problem, we propose an Averaged Stochastic Gauss-Newton algorithm, which consists in modifying equation \eqref{defsgn} by introducing a term $n^{1-\alpha}$ (with $\alpha \in (1/2,1)$), leading step sequence of the form $\frac{1}{n^{\alpha}}$. Finally, in order to ensure the asymptotic efficiency, we add an averaging step.
\begin{defi}[Averaged Stochastic Gauss-Newton algorithm]
The Averaged Stochastic Gauss Newton algorithm is recursively defined for all $n \geq 0$ by
\begin{align}
\notag \overline{\Phi}_{n+1} & = \nabla_{h} f \left( X_{n+1} , \overline{\theta}_{n} \right)  \\
\label{defsgnalpha}\theta_{n+1} & = \theta_{n} + \gamma_{n+1}\overline{S}_{n}^{-1}\left( Y_{n+1} - f \left( X_{n+1},\theta_{n} \right) \right)\nabla_{h} f \left( X_{n+1} , \theta_{n} \right)  \\
\label{asgn}\overline{\theta}_{n+1} & = \frac{n+1}{n+2}\overline{\theta}_{n} + \frac{1}{n+2}\theta_{n+1} \\
\notag S_{n+\frac{1}{2}}^{-1} & = S_{n}^{-1} - \left( 1+ \frac{c_{\beta}}{(n+1)^{\beta}} Z_{n+1}^{T}S_{n}^{-1}Z_{n+1} \right)^{-1} \frac{c_{\beta}}{(n+1)^{\beta}} S_{n}^{-1}Z_{n+1}Z_{n+1}^{T}S_{n}^{-1} \\
\notag S_{n+1}^{-1} & = S_{n}^{-1} - \left( 1+ \overline{\Phi}_{n+1}^{T}S_{n}^{-1}\overline{\Phi}_{n+1} \right)^{-1} S_{n}^{-1}\overline{\Phi}_{n+1}\overline{\Phi}_{n+1}^{T}S_{n}^{-1} .
\end{align}
where $\overline{\theta}_{0}=\theta_{0}$ is bounded, $S_{0}$ is symmetric and positive, $\gamma_{n} = c_{\alpha}n^{-\alpha}$ with $c_{\alpha}> 0,c_{\beta} \geq 0$, $\alpha \in (1/2,1)$, $\beta \in ( 0 , \alpha - 1/2)$ and $\overline{S}_{n}^{-1} = (n+1)S_{n}^{-1}$.
\end{defi}
Let us note that despite the modification of the algorithm, Riccati's formula always hold. Finally, remark that if assumption \textbf{(H*)} is satisfied, one can take $c_{\beta}=0$ and Theorem~\ref{consistency} in Section~\ref{sectiontheo} remains true. 

\subsection{Additional assumptions}
 We now suppose that these additional assumptions are fulfilled:
\begin{itemize}
	\item[\textbf{(H2)}] The functional $G$ is twice Frechet differentiable in $\mathbb{R}^{q}$ and 
there is a positive constant $C'$ such that for all $h\in \mathbb{R}^{q}$,
	\[
	\left\| \nabla^{2}G(h) \right\| \leq C' .
	\]
	
\end{itemize}
Note that this is an usual assumption in stochastic convex optimization (see for instance \citep{KY03}), and especially for studying the convergence of stochastic algorithms \citep{bach2014adaptivity,godichon2016,gadat2017optimal}. 
\begin{itemize}
	\item[\textbf{(H3)}] The function $L$ is continuous at $\theta$. 
	
\end{itemize}
Assumption \textbf{(H3)} is used for the consistency of the estimates of $L(\theta)$, which enables to give the almost sure rates of convergence of stochastic Gauss-Newton estimates and their averaged version. 

Let us now make some additional assumptions on the Hessian of the function we would like to minimize:
\begin{itemize}
	\item[\textbf{(H4a)}] The functional $h \longmapsto \nabla^{2}G(h)$ is continuous on a neighborhood of $\theta$.
	\item[\textbf{(H4b)}] The functional $h \longmapsto \nabla^{2}G(h)$ is $C_{G}$-Lipschitz on a neighborhood of $\theta$.
\end{itemize}
Assumption \textbf{(H4a)} is useful for establishing the rate of convergence of stochastic Gauss-Newton algorithms given by \eqref{defsgn} and \eqref{defsgnalpha} while assumption \textbf{(H4b)} enables to give the rate of convergence of the averaged estimates. Clearly \textbf{(H4b)} implies \textbf{(H4a)}. Note that as in the case of assumption \textbf{(H2)}, these last ones are crucial to obtain almost sure rates of convergence of stochastic gradient estimates (see for instance \citep{pelletier1998almost,godichon2016}).

\section{Convergence results}\label{sectiontheo}
We focus here on the convergence of the Averaged Stochastic Gauss-Newton algorithms since the proofs are more unusual than the ones for the non averaged version. Indeed, these last ones are quite closed to the proofs in \cite{BGBP2019}. 
\subsection{Convergence results on the averaged estimates}
The first theorem deals with the almost sure convergence of a subsequence of $\nabla G \left( \theta_{n} \right) $.  
\begin{theo}\label{consistency}
Under assumptions \textbf{(H1)} and \textbf{(H2)} with $c_{\beta}>0$, $G \left( \theta_{n} \right)$ converges almost surely to a finite random variable and there is a subsequence $(\theta_{\varphi_{n}})$ such that
\[
\left\| \nabla G \left( \theta_{\varphi_{n}} \right) \right\| \xrightarrow[n\to + \infty]{a.s} 0 
\]
\end{theo}
Remark that if the functional $G$ is convex or if we project the algorithm on a convex subspace  where $G$ is convex, previous theorem leads to the almost sure convergence of the estimates. In order to stay as general as possible, let us now introduce the event
\[
\Gamma_{\theta} = \left\lbrace \omega \in \Omega, \quad  \theta_{n}(\omega) \xrightarrow[n\to + \infty]{} \theta \right\rbrace .
\]
It is not unusual to introduce this kind of events for studying convergence of stochastic algorithms without loss of generality: see for instance \citep{pelletier1998almost, Pel00}.
Many criteria can ensure that $\mathbb{P}\left[ \Gamma_{\theta} \right] =1$, i.e that $\theta_{n}$ converges almost surely to $\theta$ (see \cite{Duf97}, \cite{KY03} for stochastic gradient descents and \cite{BGBP2019} for an example of stochastic Newton algorithm). The following corollary gives the almost sure convergence of the estimates of $L(\theta)$.
\begin{cor}\label{corconvsn}
Under assumptions \textbf{(H1)} to \textbf{(H3)}, on $\Gamma_{\theta}$, the following almost sure convergences hold:
\[
\overline{S}_{n} \xrightarrow[n\to + \infty]{a.s} H \quad \quad \text{and} \quad \quad \overline{S}_{n}^{-1} \xrightarrow[n\to + \infty]{a.s} H^{-1} .
\]
\end{cor}
In order to get the rates of convergence of the estimates $(\theta_{n})$, let us first introduce a new assumption:
\begin{itemize}
\item[\textbf{(H5)}]  There are positive constants $\eta , C_{\eta}$ such that $\eta > \frac{1}{\alpha} -1$ and for all $h \in \mathbb{R}^{q}$,
\begin{equation}
\label{momentdemerde} \mathbb{E}\left[ \left\| \nabla_{h} g \left( X, Y ,h \right) \right\|^{2 \eta +2} \right] \leq C_{\eta} .
\end{equation}
\end{itemize}
\begin{theo}\label{ratetheta}
Assume assumptions \textbf{(H1)} to \textbf{(H4a)} and \textbf{(H5)} hold. Then, on $\Gamma_{\theta}$,
\begin{equation}
\label{ratethetaas}\left\| \theta_{n} - \theta \right\|^{2} = \mathcal{O}\left( \frac{\ln n}{n^{\alpha}}\right) \quad a.s.
\end{equation}
Besides, adding assumption \textbf{(H4b)}, assuming that the function $L$ is $C_{f}$-Lipschitz on a neighborhood of $\theta$, and that the function $h \longmapsto \mathbb{E}\left[ \nabla_{h}g\left( X,Y,h \right) \nabla_{h} g \left( X,Y ,h \right)^{T} \right]$ is $C_{g}$-Lipschitz on a neighborhood of $\theta$, then on $\Gamma_{\theta}$,
\begin{equation}
\label{tlc}\sqrt{\frac{n^{\alpha}}{c_{\alpha}}} \left( \theta_{n} - \theta \right) \xrightarrow[n\to + \infty]{\mathcal{L}} \mathcal{N}\left( 0,\frac{\sigma^{2}}{2} L(\theta)^{-1} \right) 
\end{equation}
\end{theo}
These are quite usual results for a step sequence of order $n^{-\alpha}$. Indeed, as expected, we have a "loss" on the rate of convergence of $(\theta_{n})$, but the averaged step enables to get an asymptotically optimal behaviour, which is given by the following theorem.
\begin{theo}\label{ratethetabar}
Assuming \textbf{(H1)} to \textbf{(H4a)} together with \textbf{(H5)}, on $\Gamma_{\theta}$,
\[
\left\| \overline{\theta}_{n} - \theta \right\|^{2} = \mathcal{O}\left( \frac{\ln n}{n}\right) \quad a.s .
\]
Moreover, suppose that the function $h \longmapsto \mathbb{E}\left[ \nabla_{h}g\left( X,Y,h \right) \nabla_{h} g \left( X,Y ,h \right)^{T} \right]$ is continuous at $\theta$, then on $\Gamma_{\theta}$,
\[
\sqrt{n} \left( \overline{\theta}_{n} - \theta \right) \xrightarrow[n\to + \infty]{\mathcal{L}} \mathcal{N}\left( 0 , \sigma^{2}L(\theta)^{-1} \right) .
\]
\end{theo}
\begin{cor}\label{ratesn}
Assuming \textbf{(H1)} to \textbf{(H5)} and assuming  that the functional $L$ is $C_{f}$-Lipschitz on a neighborhood of $\theta$, then on $\Gamma_{\theta}$, for all $\delta > 0$,
\begin{align*}
\left\| \overline{S}_{n} - H \right\|_{F}^{2}& =\mathcal{O}\left( \max \left\lbrace \frac{c_{\beta}^{2}}{n^{2\beta}} , \frac{(\ln n)^{1+\delta}}{n} \right\rbrace \right)  a.s,  \\ \left\| \overline{S}_{n}^{-1} - H^{-1}\right\|_{F}^{2} & = \mathcal{O}\left( \max \left\lbrace \frac{c_{\beta}^{2}}{n^{2\beta}} , \frac{(\ln n)^{1+\delta}}{n} \right\rbrace \right)  a.s.
\end{align*}
\end{cor}
Recall here that if \textbf{(H*)} holds, then one can take $c_{\beta} = 0$ leading to a rate of convergence of order $\frac{1}{n}$ (up to a log term). 
\subsection{Convergence results on the Stochastic Gauss-Newton algorithm}
\begin{theo}\label{theothetatilde}
Assume assumptions \textbf{(H1)} to \textbf{(H5)} hold. Then, on $\widetilde{\Gamma}_{\theta}:= \left\lbrace \omega \in \Omega, \quad \widetilde{\theta}_{n}(\omega) \xrightarrow[n\to + \infty]{a.s} \theta \right\rbrace$,
\[
\left\| \tilde{\theta}_{n} - \theta \right\|^{2} = \mathcal{O} \left( \frac{\ln n}{n}\right) \quad a.s .
\]
Moreover, suppose that the functional $h \longmapsto \mathbb{E}\left[ \nabla_{h}g\left( X,Y,h \right) \nabla_{h} g \left( X,Y ,h \right)^{T} \right]$ is continuous at $\theta$, then on $\widetilde{\Gamma}_{\theta}$,
\[
\sqrt{n} \left( \widetilde{\theta}_{n} - \theta \right) \xrightarrow[n\to + \infty]{\mathcal{L}} \mathcal{N}\left( 0 , \sigma^{2}L(\theta)^{-1} \right) .
\]
\end{theo}

\subsection{Estimating the variance $\sigma^{2}$}
We now focus on the estimation of the variance $\sigma^{2}$ of the errors. First, consider the sequence of predictors $\left( \hat{Y}_{n} \right)_{n\geq 1}$  of $Y$ defined for all $n \geq 1$ by
\[
\hat{Y}_{n} = f \left( X_{n} , \overline{\theta}_{n-1}\right).
\]
Then, one can estimate $\sigma^{2}$ by the recursive estimate $\hat{\sigma}_{n}$ defined for all $n \geq 1$ by
\[
\hat{\sigma}_{n} = \frac{1}{n} \sum_{k=1}^{n} \left( \hat{Y}_{k} - Y_{k} \right)^{2}. 
\]
\begin{cor}\label{estsigma}
Assume assumptions \textbf{(H1)} to \textbf{(H5)} hold and that $\epsilon$ admits a $4$-th order moment. Then, on $\Gamma_{\theta}$,
\[
\left| \hat{\sigma}_{n}^{2} - \sigma^{2} \right| = \mathcal{O} \left( \sqrt{\frac{\ln n}{n}} \right) \quad a.s., 
\]
and 
\[
\sqrt{n} \left( \hat{\sigma}_{n}^{2} - \sigma^{2} \right) \xrightarrow[n\to + \infty]{\mathcal{L}} \mathcal{N}\left( 0 , \mathbb{V}\left[ \epsilon^{2} \right] \right).
\]
\end{cor}

\section{Simulation study}\label{sectionsimu}
In this section, we present a short simulation study to illustrate convergence results 
given in Section \ref{sectiontheo}.
Simulations were carried out using the statistical software \Rlogo.

\subsection{The model}
Consider the following model 
\[
Y = \theta_1 \left( 1- \exp (- \theta_2 X)\right) +\epsilon, 
\]
with $\theta = (\theta_1, \theta_2)^T = (21,12)^T$, $X \sim \mathcal{U}([0,1])$, and $\epsilon \sim \mathcal{N}( 0, 1)$. 
For sure, this model is very simple. 
However, it allows us to explore 
the behaviour of the different algorithms presented in this paper by comparing our methods with the stochastic gradient algorithm and its averaged version and by looking at the influence of the step sequence on the estimates. 
In that special case, matrix $L(\theta)=\nabla^2 G(\theta)$ is equal to:
\[
L ( \theta) = \mathbb{E}\left[ \begin{pmatrix}
\left( 1- \exp \left( - \theta_{2}X \right) \right)^{2} & \theta_{1}X \left( 1- \exp \left( - \theta_{2}X \right) \right) \exp \left( - \theta_{2}X \right) \\
\theta_{1}X \left( 1- \exp \left( - \theta_{2}X \right) \right) \exp \left( - \theta_{2}X \right) & \theta_{1}^{2}X^{2} \exp \left( - 2 \theta_{2}X\right) 
\end{pmatrix} \right] .
\]
Then, taking $\theta = ( 21 , 12 )^T$, we obtain
\[
L ( 21 , 12) \simeq \begin{pmatrix}
0.875 & 0.109 \\
0.109 & 0.063
\end{pmatrix}
\]
and in particular, the eigenvalues are about equal (in decreasing order) to $0.889$ and $0.049$. 

\subsection{Comparison with stochastic gradient algorithm and influence of the step-sequence}
To avoid  some computational problems,
we consider projected versions of the different algorithms: more precisely, each algorithm is projected on the ball of center $(21, 12)^T$ and of radius 12.

For each algorithm, we calculate the mean squared error
based on 100  independent samples of size $n=10\ 000$, with initialization $\theta_{0} = \theta + 10 U$, where $U$ follows an uniform law on the unit sphere of $\mathbb{R}^{2}$. 

We can see in Tables \ref{table1} and \ref{tablegrad2}
that Gauss-Newton methods perform globally better than gradient descents. 
This is certainly due to the fact that the eigenvalues of $L(\theta)$ are at quite different scales, 
so that the step sequence in gradient descent is less adapted to each direction. Remark that we could have been less fair play, considering an example where the eigenvalues of $L(\theta)$ would have been at most sensitively different scales. 
Nevertheless, without surprise, the estimates seem to be quite sensitive to the choices of the parameters $c_{\alpha}$ and $\alpha$.
\begin{table}[H] 
\centering
\small{\begin{tabular}{r|rrrr}

 \backslashbox{$c_{\alpha}$}{$\alpha$}  & 0.55 & 0.66 & 0.75  & 0.9 \\ 
  \hline
 0.1 & 0.0241 & 0.3873 & 2.1115 & 9.1682 \\ 
  0.5 & 0.1229 & 0.0130 & 0.0057 & 0.0171 \\ 
   1 & 0.0732 & 0.0257 & 0.0113 & 0.0030 \\ 
   2 & 0.1203 & 0.0503 & 0.0213 & 0.0073 \\ 
   5 & 0.4029 & 0.1647 & 0.0762 & 0.0168 \\ 
 \end{tabular}  \quad \begin{tabular}{r|rrrr}
 
\backslashbox{$c_{\alpha}$}{$\alpha$}  & 0.55 & 0.66 & 0.75 & 0.9 \\ 
  \hline
 0.1 & 0.3280 & 1.3746 & 4.1304 & 13.5156 \\ 
   0.5 & 0.1393 & 0.0078 & 0.0135 & 0.0928 \\ 
  1 & 0.0068 & 0.0049 & 0.0085 & 0.0440 \\ 
  2 & 0.0104 & 0.0058 & 0.0052 & 0.0082 \\ 
   5 & 0.2076 & 0.0260 & 0.0067 & 0.0027 \\ 
\end{tabular}}
\caption{\label{table1}
 Mean squared errors of Stochastic Newton estimates defined by \eqref{defsgnalpha} (on the left) and its averaged version defined by \eqref{asgn} (on the right) for different parameters $\alpha$ and $c_{\alpha}$, for $n=10\ 000$. }
\end{table}

\begin{table}[H] 
\centering
\small{\begin{tabular}{r|rrrr}
 
\backslashbox{$c_{\alpha}$}{$\alpha$}  & 0.55 & 0.66 & 0.75  & 0.9 \\ 
  \hline
 0.1 & 15.40 & 24.69 & 35.32 & 40.42 \\ 
   0.5 & 1.30 & 8.36 & 15.90 & 22.80 \\ 
   1 & 0.24 & 2.66 & 9.76 & 28.52 \\ 
   2 & 0.29 & 0.06 & 3.02 & 22.64 \\ 
   5 & 0.26 & 0.02 & 0.01 & 3.89 \\ 
\end{tabular}  \quad \begin{tabular}{r|rrrr}
\backslashbox{$c_{\alpha}$}{$\alpha$}  & 0.55 & 0.66 & 0.75 & 0.9 \\ 
  \hline
 0.1 & 19.78 & 28.03 & 38.49 & 43.01 \\ 
 0.5 & 4.49 & 12.48 & 18.79 & 24.20 \\ 
  1 & 1.16 & 7.28 & 13.68 & 30.84 \\ 
   2 & 0.41 & 1.75 & 7.46 & 26.90 \\ 
  5 & 0.27 & 0.02 & 0.18 & 7.45 \\ 
\end{tabular}}
\caption{\label{tablegrad2} Mean squared errors of Stochastic Gradient estimates with  step $c_{\alpha}n^{-\alpha}$ (on the left) and its averaged version (on the right) for different parameters $\alpha$ and $c_{\alpha}$, for $n=10\ 000$. }
\end{table}
To complete this study, let us now examine 
the behaviour of the mean squared error in function of the sample size
  for the standard Stochastic Gauss-Newton algorithm ("SN"), and its averaged version ("ASN"), for the Stochastic Gauss-Newton algorithm defined by \eqref{defsgnalpha} ("ND") 
   and for the stochastic gradient algorithm ("SGD") as well as its averaged version ("ASGD"). The stochastic Gauss-Newton  algorithms, has been computed with $c_{\beta}=0$ and $S_0 = I_2$.
For the averaged stochastic Newton algorithm  a step sequence of the form $c_{\alpha}n^{-\alpha}$ with $c_{\alpha}=1$ and $\alpha = 0.66$ has been chosen, while we have taken $c_{\alpha} = 5$ and $\alpha = 0.66$ for the stochastic gradient algorithm and its averaged version.
Finally, we have considered three different initializations: $\theta_{0} = \theta + r_{0}Z$, where $Z $ follows an uniform law on the unit sphere of $\mathbb{R}^{2}$ and  $r_{0} = 1,5,12$. 

\begin{figure}[H]\centering
\includegraphics[scale=1]{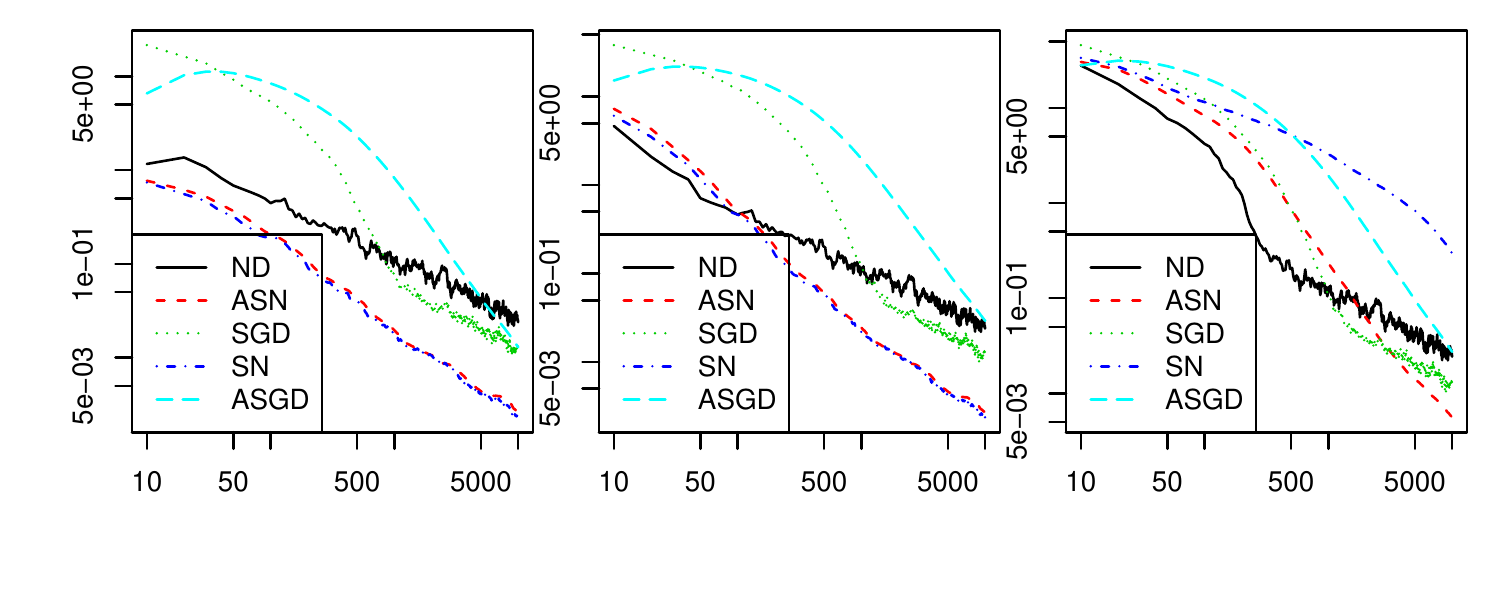}
\caption{Evolution of the mean squared error in relation with the sample size with, from the left to the right, $r_{0} = 1,5,12$. \label{newtonplot}}
\end{figure}

One can see in Figure \ref{newtonplot}
that stochastic Gauss-Newton algorithms perform better than gradient descents and their averaged version for quite good initialization. 
As explained before, this is due to the fact that stochastic Gauss-Newton methods enable to adapt the step sequence to each direction. Furthermore, one can see that when we have a quite good initialization, Stochastic Gauss-Newton and the averaged version give similar results. Nevertheless, when we deal with a bad initialization, to take a step sequence of the form $c_{\alpha}n^{-\alpha}$ enables the estimates to move faster, so that the averaged Gauss Newton estimates have better results compare to the non averaged version. Finally, on Figure~\ref{figboxplot}, one can see that Gauss-Newton methods globally perform better than gradient methods, but there are some bad trajectories due to bad initialization. 

\begin{figure}[H]\centering
\includegraphics[scale=0.8]{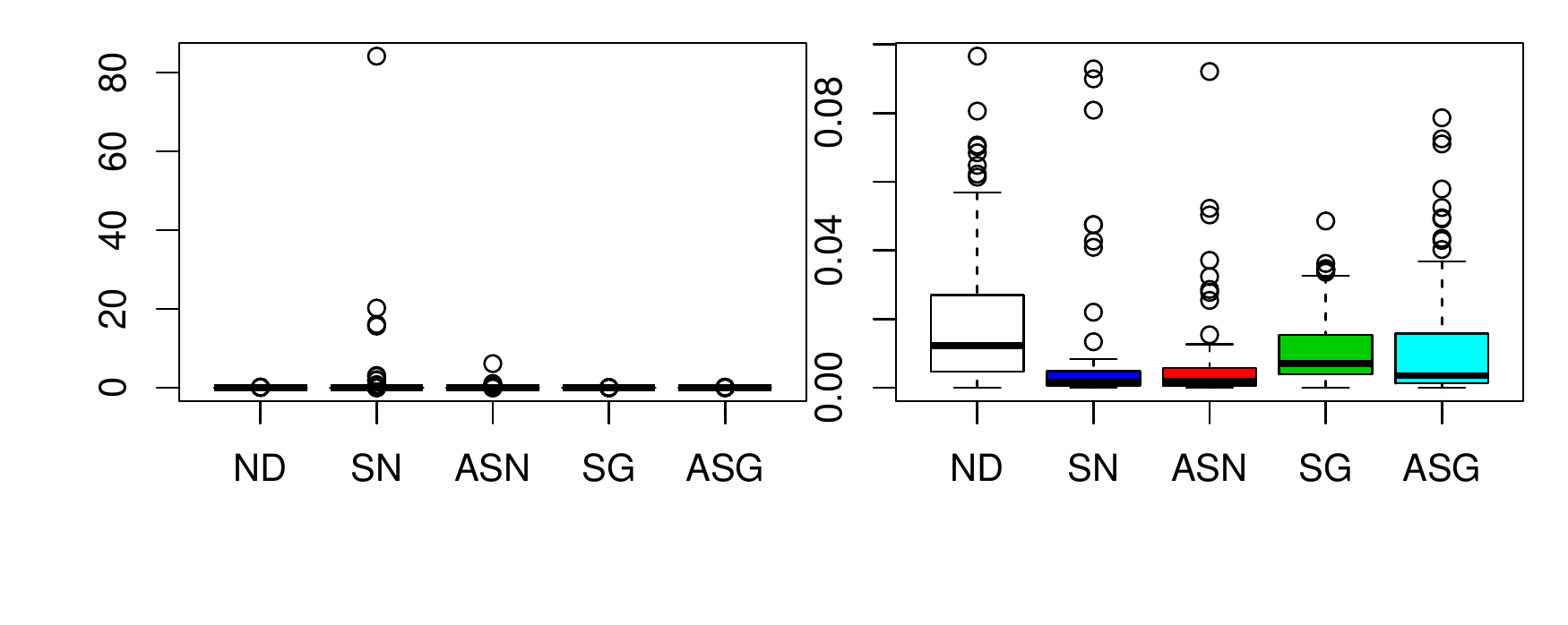}
\caption{Mean squared errors for the different methods  for $n=10\ 000$ and $r_{0} = 12$. \label{figboxplot}}
\end{figure}

\subsection{Influence of $\beta_{n}$}
We now focus on the influence of $\beta_n$ on the behaviour of the estimates. All the results are calculated from $100$ independent samples of size $n$. 
In table~\ref{tablebeta}, we have chosen $c_{\alpha}=0.66$ and $c=1$, and initialize the algorithms taking $\theta_{0} = \theta + 5 U$. Choosing small $c_{\beta}$ leads to goods results, which, without surprise, suggests that the use of the term $\sum_{k=1}^{n}\beta_{k}Z_{k}Z_{k}^{T}$ is only theoretical, but has no use in practice.
\begin{table}[H]
\centering
\scalebox{0.92}{
\scriptsize{\begin{tabular}{r|rrrr}
 \backslashbox{$c_{\beta}$}{$\beta$}  & 0.01 & 0.08 & 0.2 & 0.5 \\ 
  \hline
 $10^{-10}$ & 2.59 & 3.12 & 2.84 & 2.66 \\ 
 $10^{-5}$ & 2.59 & 2.90 & 2.42 & 1.88 \\ 
 $10^{-2}$ & 2.11 & 2.32 & 2.23 & 1.83 \\ 
 $10^{-1}$ & 0.69 & 1.09 & 1.98 & 2.64 \\ 
 $1$ & 4.04 & 0.61 & 0.66 & 2.06 \\ 
\end{tabular}
 \quad  \begin{tabular}{r|rrrr}
 \backslashbox{$c_{\beta}$}{$\beta$}  & 0.01 & 0.08 & 0.2 & 0.5 \\ 
  \hline
 $10^{-10}$ & 0.24 & 0.22 & 0.30 & 0.17 \\ 
 $10^{-5}$ & 0.29 & 0.26 & 0.20 & 0.23 \\ 
  $10^{-2}$ & 0.22 & 0.26 & 0.23 & 0.32 \\ 
   $10^{-1}$ & 0.25 & 0.24 & 0.21 & 0.27 \\ 
  $1$ & 26.37 & 11.19 & 1.46 & 0.30 \\ 
\end{tabular} \quad \begin{tabular}{r|rrrr}
 \backslashbox{$c_{\beta}$}{$\beta$} & 0.01 & 0.08 & 0.2 & 0.5 \\ 
  \hline
 $10^{-10}$ & 0.22 & 0.21 & 0.27 & 0.19 \\ 
 $10^{-5}$ & 0.27 & 0.24 & 0.19 & 0.21 \\ 
 $10^{-2}$ & 0.25 & 0.23 & 0.22 & 0.28 \\ 
 $10^{-1}$ & 11.77 & 3.89 & 0.85 & 0.26 \\ 
 $1$ & 444.61 & 409.27 & 184.09 & 9.77 \\ 
\end{tabular}
}}
\caption{Mean squared errors ($.10^{-2}$) of, from left to right, Stochastic Newton estimates with $c_{\alpha}n^{-\alpha}$ defined by \eqref{defsgnalpha}, its averaged version defined by \eqref{asgn} and Stochastic Newton estimates definded by \eqref{defsgn} for different values of $\beta$ and $c_{\beta}$. }
\label{tablebeta} 
\end{table}

\subsection{Asymptotic efficiency}

We now illustrate the asymptotic normality of the estimates. 
In order to consider all the components of parameter $\theta$,
we shall  in fact examine
the two following central limit theorem.
\[
C_n := \left( \widetilde{\theta}_{n}  - \theta \right)^T\widetilde{H}_n \left( \widetilde{\theta}_n - \theta \right) \xrightarrow[n\to + \infty]{\mathcal{L}} \chi_2^2 \quad \quad \text{and} \quad \quad \overline{C}_n := \left( \overline{\theta}_{n}  - \theta \right)^{T}S_{n} \left( \overline{\theta}_{n} - \theta \right) \xrightarrow[n\to + \infty]{\mathcal{L}} \chi_{2}^{2}.
\] 
These two results are straightforward applications of Theorems \ref{ratethetabar} and \ref{theothetatilde}, taking $\sigma^2 = 1$.

From $1000$ independent samples of size $n$,
and using the kernel method, we estimate the probability density function (pdf for short)  of $C_n$ and $\overline{C}_n$
that we compare to the pdf of a chi-squared with 2 degrees of freedom (df).
Estimates $\widetilde{\theta}_{n}$ and $\overline{\theta}_n$ are computed taking
$\theta_{0} = \theta + U$, $c_{\alpha} = 1$ and $\alpha = 0.66$.  

One can observe in Figure~\ref{chisq} that the estimated densities are quite similar to the   chi-squared pdf. Remark that a Kolmogorov-Smirnov test leads to $p$-values equal to $0.1259$ for $C_{n}$ and $0.33$ for $\overline{C}_{n}$. Therefore, the gaussian approximation provided by the TLC given by \ref{ratethetabar} and \ref{theothetatilde} is pretty good, and so, even for quite moderate sample sizes.

\begin{figure}[H]\centering
\includegraphics[scale=0.6]{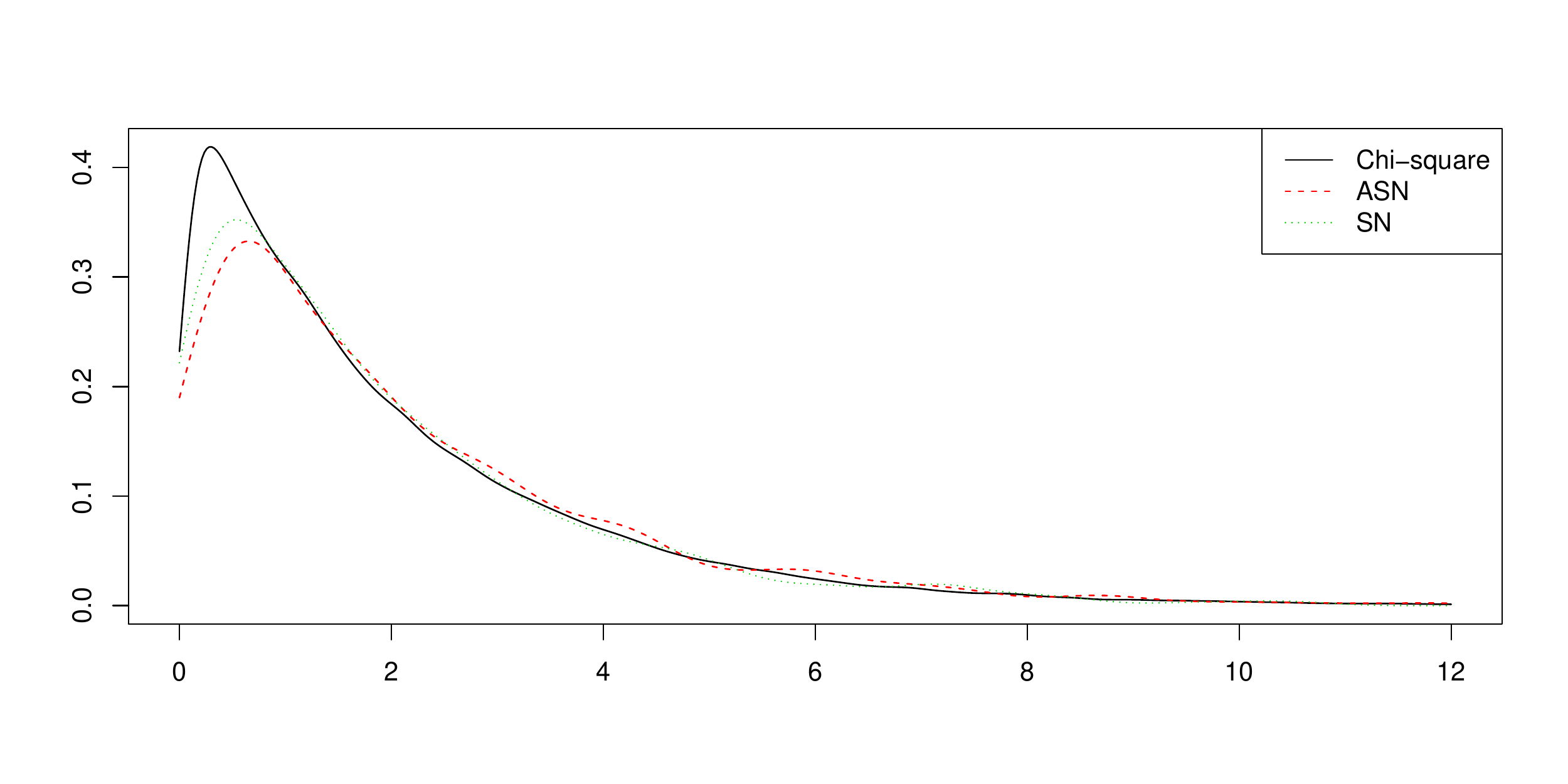}
\caption{Density function of a Chi squared with 2 df (in black)
and kernel-based density estimates of $C_{n}$ (in green) and $\overline{C}_{n}$ (in red) for $n= 5000$.}
\label{chisq} 
\end{figure}

\section{Proofs}\label{sectionproof}
\subsection{Proofs of Theorem \ref{consistency} and Corollary \ref{corconvsn}}

\begin{lem}\label{superlem}
Assuming \textbf{(H1)} and \textbf{(H2)} and taking $c_{\beta}> 0$, it comes
\[
\lambda_{\max} \left( S_{n}^{-1} \right) = \mathcal{O} \left( n^{\beta-1} \right) \quad a.s \quad \text{and} \quad \lambda_{\max} \left( S_{n} \right) =\mathcal{O}(n)\quad  a.s.
\]
\end{lem}

\begin{proof}[Proof of Lemma \ref{superlem}]
First, note that 
\[
\lambda_{\min}\left( S_{n} \right) \geq \lambda_{\min} \left(\sum_{i=1}^{n} \frac{c_{\beta}}{i^{\beta}} Z_i Z_i^{T} \right),
\]
and one can check that 
\begin{equation}
\label{vitzi}
\frac{1}{\sum_{i=1}^{n}\frac{c_{\beta}}{i^{\beta}}} \sum_{i=1}^{n} \frac{c_{\beta}}{i^{\beta}}Z_{i}Z_{i}^{T} \xrightarrow[n\to + \infty]{a.s} I_{q} .
\end{equation}
Since 
\[\sum_{i=1}^{n} \frac{c_{\beta}}{i^{\beta}} \sim \frac{c_{\beta}}{1-\beta}n^{1-\beta},\] it comes
\[
\lambda_{\max} \left( S_{n}^{-1} \right) = \mathcal{O} \left( n^{\beta -1 } \right) \quad a.s.
\]
Let us now give a bound of $\lambda_{\max}\left( S_{n} \right)$. First, remark that $S_{n}$ can be written as
\begin{equation}
\label{decsn} S_{n} = S_{0} + \sum_{i=1}^{n} \mathbb{E}\left[ \overline{\Phi}_{i}\overline{\Phi}_{i}^{T} |\mathcal{F}_{i-1} \right] + \sum_{i=1} \Xi_{i} + \sum_{i=1}^{n} \frac{c_{\beta}}{i^{\beta}}Z_{i}Z_{i}^{T},
\end{equation}
where $\Xi_{i} := \overline{\Phi}_{i}\overline{\Phi}_{i}^{T} - \mathbb{E}\left[ \overline{\Phi}_{i}\overline{\Phi}_{i}^{T} |\mathcal{F}_{i-1} \right]$ and $\left( \mathcal{F}_{i} \right)$ is the $\sigma$-algebra generated by the sample, i.e $\mathcal{F}_{i} := \sigma \left( \left( X_{1}, Y_{1} \right) , \ldots ,\left( X_{i} , Y_{i} \right) \right)$. Thanks to assumption \textbf{(H1b)}, $\mathbb{E}\left[ \left\| \overline{\Phi}_{i}\overline{\Phi}_{i}^{T} \right\|_{F}^{2} |\mathcal{F}_{i-1} \right] \leq C''$, and applying Theorem~\ref{theomartmoy}, it comes for all positive constant $\delta > 0$,
\begin{equation}\label{vitXi}
\left\| \sum_{i=1}^{n} \Xi_{i} \right\|_{F}^{2} = o \left( n (\ln n)^{1+\delta} \right) \quad a.s,
\end{equation}
where $\left\| . \right\|_{F}$ is the Frobenius norm. Then,
\[
\left\| S_{n} \right\|_{F} = \mathcal{O} \left( \max \left( \left\| S_{0} \right\|_{F}, C''n , \sqrt{n (\ln n)^{1+\delta}} , n^{1-\beta} \right)\right) \quad a.s,  
\]
which concludes the proof.
\end{proof}

\begin{proof}[Proof of Theorem \ref{consistency}]
 With the help of a Taylor's decomposition, it comes
\[
G \left( \theta_{n+1} \right) = G \left( \theta_{n} \right) + \nabla G \left( \theta_{n} \right)^{T}\left( \theta_{n+1} - \theta_{n} \right) + \frac{1}{2}\left( \theta_{n+1} - \theta_{n} \right)^{T} \int_{0}^{1} \nabla^{2} G \left( \theta_{n+1} + t \left( \theta_{n+1} - \theta_{n} \right) \right) dt \left( \theta_{n+1} - \theta_{n} \right) .
\]
Then, assumption \textbf{(H2)} yields
\[
G \left( \theta_{n+1} \right) \leq G \left( \theta_{n} \right) + \nabla G \left( \theta_{n} \right)^{T}\left( \theta_{n+1} - \theta_{n} \right) + \frac{C'}{2}\left\| \theta_{n+1} - \theta_{n} \right\|^{2}.
\]
Replacing $\theta_{n+1}$,
\begin{align*}
G & \left( \theta_{n+1} \right)\leq G \left( \theta_{n} \right) - \gamma_{n+1}\nabla G \left( \theta_{n} \right)^{T}\overline{S}_{n}^{-1}\nabla_{h} g \left( X_{n+1} , Y_{n+1} , \theta_{n} \right) + \frac{C'}{2}\gamma_{n+1}^{2} \left\| \overline{S}_{n}^{-1} \nabla_{h} g \left( X_{n+1} , Y_{n+1} ,\theta_{n} \right) \right\|^{2}  \\
& \leq G \left( \theta_{n} \right) - \gamma_{n+1} \nabla G \left( \theta_{n} \right)^{T} \overline{S}_{n}^{-1} \nabla_{h} g \left( X_{n+1} , Y_{n+1} , \theta_{n} \right) + \frac{C'}{2}\gamma_{n+1}^{2} \left( \lambda_{\max}  \left( \overline{S}_{n}^{-1} \right) \right)^{2} \left\|  \nabla_{h} g \left( X_{n+1} , Y_{n+1} ,\theta_{n} \right) \right\|^{2}.
\end{align*}
Assumption \textbf{(H1a)} leads to
\begin{align*}
\mathbb{E}\left[ G \left( \theta_{n+1} \right) |\mathcal{F}_{n} \right] & \leq G \left( \theta_{n} \right) - \gamma_{n+1} \nabla G \left( \theta_{n} \right)^{T} \overline{S}_{n}^{-1} \nabla G \left( \theta_{n} \right)  + \frac{CC'}{2}\gamma_{n+1}^{2} \left( \lambda_{\max}  \left( \overline{S}_{n}^{-1} \right) \right)^{2} \\
& \leq G \left( \theta_{n} \right) - \gamma_{n+1}\lambda_{\min} \left( \overline{S}_{n}^{-1} \right) \left\| \nabla G \left( \theta_{n} \right) \right\|^{2} + \frac{CC'}{2}\gamma_{n+1}^{2}\left( \lambda_{\max}  \left( \overline{S}_{n}^{-1} \right) \right)^{2}.
\end{align*}
Thanks to Lemma \ref{superlem}, and since $\beta < \alpha - 1/2$, 
\[
\sum_{n\geq 1} \gamma_{n+1}^{2}\left(  \lambda_{\max} \left (\overline{S}_{n}^{-1} \right) \right)^{2} < + \infty \quad a.s,
\]
so that, applying Robbins-Siegmund Theorem (see \cite{Duf97} for instance), $G \left( \theta_{n} \right)$ converges almost surely to a finite random variable and
\[
\sum_{n\geq 1} \gamma_{n+1} \lambda_{\min} \left( \overline{S}_{n}^{-1} \right)\left\| \nabla G \left( \theta_{n} \right) \right\|^{2} = \sum_{n\geq 1 } \gamma_{n+1} \lambda_{\max} \left( \overline{S}_{n} \right)^{-1} \left\| \nabla G \left( \theta_{n} \right) \right\|^{2} < + \infty \quad a.s.
\]
Lemma~\ref{superlem} implies $\sum_{n\geq 1} \gamma_{n+1} \lambda_{\max} \left( \overline{S}_{n}\right)^{-1} = + \infty$ almost surely, so that there is alsmot surely a subsequence $\theta_{\varphi_{n}}$ such that $\left\| \nabla G \left( \theta_{\varphi_{n}} \right) \right\|^{2} $ converges to $0$.

\end{proof}

\begin{proof}[Proof of Corollary \ref{corconvsn}] Let us give the convergence of each term in decomposition (\ref{decsn}). Since $\beta > 0$ and applying (\ref{vitXi}), it comes
\[
\frac{1}{n}\left\| \sum_{i=1}^{n}\Xi_{i} \right\|_{F} \xrightarrow[n\to + \infty]{a.s} 0 \quad a.s  \quad \text{and} \quad \frac{1}{n}\left\| \sum_{i=1}^{n} \frac{c_{\beta}}{i^{\beta}}Z_{i}Z_{i}^{T} \right\|_{F} \xrightarrow[n\to + \infty]{a.s} 0 \quad a.s.
\]
Finally, since $\overline{\theta}_{n}$ converges almost surely to $\theta$, assumption \textbf{(H3)} together with Toeplitz lemma yield
\[
\frac{1}{n}\sum_{i=1}^{n} \mathbb{E}\left[ \overline{\Phi}_{i}\overline{\Phi}_{i}^{T} |\mathcal{F}_{i-1} \right] \xrightarrow[n\to + \infty]{a.s} \mathbb{E}\left[ \nabla_{h}f \left( X, \theta \right) \nabla_{h} f\left( X , \theta \right)^{T} \right] = L(\theta).
\]
\end{proof}

\subsection{Proof of Theorem \ref{ratetheta}}
First, $\theta_{n+1}$ can be written as
\begin{equation}
\label{decxi} \theta_{n+1} - \theta = \theta_{n} - \theta -\gamma_{n+1}\overline{S}_{n}^{-1} \nabla G \left( \theta_{n} \right) + \gamma_{n+1}\overline{S}_{n}^{-1}\xi_{n+1},
\end{equation}
with $\xi_{n+1} :=\nabla G\left( \theta_{n} \right) - \nabla_{h} g \left( X_{n+1} , Y_{n+1} , \theta_{n} \right)$. Remark that $\left( \xi_{n} \right)$ is a sequence of martingale differences adapted to the filtration $\left( \mathcal{F}_{n} \right)$. Moreover, linearizing the gradient and noting $H = \nabla^{2}G(\theta)$, it comes
\begin{align}
\label{decareutiliser} \theta_{n+1} - \theta & = \theta_{n} - \theta - \gamma_{n+1}\overline{S}_{n}^{-1} H \left( \theta_{n} - \theta \right) + \gamma_{n+1}\overline{S}_{n}^{-1}\xi_{n+1} - \gamma_{n+1}\overline{S}_{n}^{-1}\delta_{n} \\
\notag & = \left( 1-\gamma_{n+1} \right)\left( \theta_{n} - \theta \right) + \gamma_{n+1}\left( H^{-1} -\overline{S}_{n}^{-1} \right) H \left( \theta_{n} - \theta \right)\\
\label{decdelta} & + \gamma_{n+1}\overline{S}_{n}^{-1}\xi_{n+1} -\gamma_{n+1}\overline{S}_{n}^{-1}\delta_{n},
\end{align}
where $\delta_{n} = \nabla G \left( \theta_{n} \right) - H \left( \theta_{n} - \theta \right) $ is the remainder term in the Taylor's decomposition of the gradient. By induction, one can check that for all $n \geq 1$,
\begin{align}
\notag \theta_{n} - \theta & = \beta_{n,0} \left( \theta_{0} - \theta \right) +  \sum_{k=0}^{n-1} \beta_{n,k+1}\gamma_{k+1} \left( H^{-1} - \overline{S}_{k}^{-1}\right) H \left( \theta_{k} - \theta \right) -  \sum_{k=0}^{n-1} \beta_{n,k+1}\gamma_{k+1}\overline{S}_{k}^{-1}\delta_{k} \\
\label{decbeta} &  +  \sum_{k=0}^{n-1} \beta_{n,k+1} \gamma_{k+1}\overline{S}_{k}^{-1}\xi_{k+1} ,
\end{align}
with, for all $k,n \geq 0$ and $k \leq n$, \[\beta_{n,k} = \prod_{j=k+1}^{n} \left( 1-\gamma_{j} \right) \quad \mbox{and} \quad \beta_{n,n}=1.\]
Applying Theorem~\ref{theomartbeta},
\begin{equation}
\label{lembeta}\left\| \sum_{k=0}^{n-1}\beta_{n,k+1}\gamma_{k+1}\overline{S}_{k}^{-1} \xi_{k+1} \right\|^{2} = \mathcal{O} \left( \frac{\ln n}{n^{\alpha}} \right) \quad a.s.
\end{equation}
We can now prove Theorem~\ref{ratetheta}.
\begin{proof}[Proof of equation~\ref{ratethetaas}]
The aim is to give the rate of convergence of each term of decomposition (\ref{decbeta}). The rate of the martingale term is given  by equation \eqref{lembeta}.  Remark that there is a rank $n_{\alpha}$ such that for all $n \geq n_{\alpha}$, we have $\gamma_{n}\leq 1$, so that, for all $n \geq n_{\alpha}$,
\begin{align*}
\left\| \beta_{n,0}\left( \theta_{0} - \theta \right) \right\| & \leq \left\| \theta_{0} - \theta \right\| \prod_{i=1}^{n_{\alpha}-1} \left| 1- \gamma_{i+1} \right| \prod_{i=n_{\alpha}} \left( 1-\gamma_{i+1} \right) \\
& \leq \left\| \theta_{0} - \theta \right\| \prod_{i=1}^{n_{\alpha}-1} \left| 1- \gamma_{i+1} \right| \exp \left( -\sum_{i=n_{\alpha}}^{n}\gamma_{i+1} \right) .
\end{align*}
Since $\alpha < 1$, this term converges at an exponential rate, and more precisely
\begin{equation}\label{vitexp}
\left\| \beta_{n,0}\left( \theta_{0} - \theta \right) \right\| = \mathcal{O} \left( \exp \left( - \frac{c_{\alpha}}{1-\alpha}n^{1-\alpha} \right) \right)  \quad a.s.
\end{equation}
Let us denote
\[
\Delta_{n} :=  \sum_{k=0}^{n-1}\beta_{n,k+1}\gamma_{k+1}\left( H^{-1} - \overline{S}_{k}^{-1} \right) H \left( \theta_{k} - \theta \right) -  \sum_{k=0}^{n} \beta_{n,k+1}\gamma_{k+1} \overline{S}_{k}^{-1}\delta_{k} .
\]
The aim is so to prove that this term is negligible. First, remark that the Taylor's decomposition of the gradient yields
\[
\delta_{n} = \int_{0}^{1} \left( \nabla^{2}G \left( \theta + t \left( \theta_{n} - \theta \right) \right) - H \right) \left( \theta_{n} - \theta \right) dt ,
\]
and thanks to assumption \textbf{(H4a)}, since $\theta_{n}$ converges almost surely to $\theta$ and by dominated convergence,
\[
\left\| \delta_{n} \right\| \leq \left\| \theta_{n} - \theta \right\| \int_{0}^{1} \left\| \nabla^{2}G \left( \theta + t \left( \theta_{n} - \theta \right) \right) - H \right\|_{op} dt = o \left( \left\| \theta_{n} - \theta \right\| \right) \quad a.s
\] 
In a similar way, since $\overline{S}_{n}^{-1}$ converges almost surely to $H^{-1}$,
\[
\left\| \left( H^{-1} - \overline{S}_{n}^{-1} \right) H \left( \theta_{n} - \theta \right) \right\| = o \left( \left\| \theta_{n} - \theta \right\| \right) \quad a.s.
\]
Moreover, since
\[
\Delta_{n+1} = \left( 1-\gamma_{n+1} \right) \Delta_{n} + \gamma_{n+1} \left( H^{-1} - \overline{S}_{n}^{-1} \right) H \left( \theta_{n} - \theta \right) - \gamma_{n+1}\overline{S}_{n}^{-1}\delta_{n},
\]
we have, 
\begin{align*}
\left\| \Delta_{n+1} \right\| & \leq \left| 1- \gamma_{n+1} \right| \left\| \Delta_{n} \right\| + o \left( n^{-\alpha} \left\| \theta_{n} - \theta \right\| \right) \quad a.s \\
& \leq \left| 1- \gamma_{n+1} \right| \left\| \Delta_{n} \right\| + o \left( n^{-\alpha} \left\| \Delta_{n} + \sum_{k=0}^{n-1} \beta_{n,k+1} \gamma_{k+1} \overline{S}_{k}^{-1} \xi_{k+1} + \beta_{n,0}\left( \theta_{0} - \theta \right) \right\| \right) \quad a.s.
\end{align*}
Then, since $\beta_{n,0}\left( \theta_{0} - \theta \right)$ converges at an exponential rate to $0$ and thanks to equation \eqref{lembeta}, 
\[
\left\| \Delta_{n+1} \right\| \leq \left( 1- \gamma_{n+1} \right) \left\| \Delta_{n} \right\| + o \left( \gamma_{n+1} \sqrt{\frac{\ln n}{n^{\alpha}}} + \gamma_{n+1}\left\| \Delta_{n} \right\| \right) \quad a.s ,
\]
and applying a stabilization Lemma (see \cite{Duf97} for instance),
\[
\left\| \Delta_{n} \right\| = \mathcal{O} \left( \sqrt{\frac{\ln n}{n^{\alpha}}}\right) \quad a.s
\]
which concludes the proof.
\end{proof}

\begin{proof}[Proof of equation \eqref{tlc}]
In order to get the asymptotic normality, we will apply the Central Limit Theorem in \cite{Jak88} to the martingale term in decomposition (\ref{decbeta}) and prove that other terms of this decomposition are negligible. The rate of convergence of the first term on the right-hand side of equality (\ref{decbeta}) is given by equality (\ref{vitexp}). Moreover, applying Theorem \ref{ratetheta} and Remark \ref{rempourletrucpourri} as well as Lemma E.2 in \cite{CG2015}, one can check that for all $\delta > 0$,
\[
\left\| \sum_{k=0}^{n-1}\beta_{n,k+1}\gamma_{k+1}\left( H^{-1} - \overline{S}_{k}^{-1} \right) H \left( \theta_{k} - \theta \right) \right\|^{2} = o \left( \frac{(\ln n)^{1+\delta}}{n^{ \alpha + \beta}} \right) \quad a.s,
\] 
and this term is so negligible. In the same way, 
\[
\left\| \sum_{k=0}^{n-1}\beta_{n,k+1}\gamma_{k+1}\overline{S}_{k}^{-1} \delta_{k} \right\|^{2} = o \left( \frac{(\ln n)^{2+ \delta}}{n^{ 2\alpha }} \right) \quad a.s.
\]
Note that the martingale term can be written as
\[
\sum_{k=0}^{n-1} \beta_{n,k+1}\gamma_{k+1}\overline{S}_{k}^{-1}\xi_{k+1}= \sum_{k=0}^{n-1} \beta_{n,k+1} c_{\alpha}\gamma_{k+1} \left(  \overline{S}_{k}^{-1} - H^{-1} \right) \xi_{k+1} + \sum_{k=0}^{n-1} \beta_{n,k+1}\gamma_{k+1} H^{-1}\xi_{k+1} . 
\]
Applying Theorem \ref{theomartbeta} and Remark \ref{rempourletrucpourri},
\[
\left\| \sum_{k=0}^{n} \beta_{n+1,k+1} \gamma_{k+1} \left(  \overline{S}_{k}^{-1} - H^{-1} \right) \xi_{k+1} \right\|^{2} = O \left( \frac{\ln n}{n^{\alpha + 2\beta}} \right) \quad a.s.
\]
Finally, let us now prove that the martingale term verifies assumption in \cite{Jak88}, i.e that we have
\begin{equation}
\label{cond1} \forall \nu > 0 , \quad \lim_{n\to + \infty} \mathbb{P}\left[ \sup_{0\leq k \leq n} \sqrt{\frac{n^{\alpha}}{c_{\alpha}}} \left\| \beta_{n+1,k+1} c_{\alpha}(k+1)^{-\alpha}H^{-1} \xi_{k+1} \right\| > \nu \right] = 0 .
\end{equation}
\begin{equation}
\label{cond2} \frac{n^{\alpha}}{c_{\alpha}} \sum_{k=0}^{n} \beta_{n+1,k+1}^{2} \gamma_{k+1}^{2} H^{-1} \xi_{k+1} \xi_{k+1}^{T} H^{-1} \xrightarrow[n\to + \infty]{a.s} \frac{1}{2}\sigma^{2} H^{-1} .
\end{equation}
\medskip

\noindent\textbf{Proof of equation (\ref{cond1}):}  Thanks to assumption \textbf{(H5)}, one can check that for all $n\geq 1$,
\begin{equation}
\label{majxipourri}\mathbb{E}\left[ \left\| \xi_{n+1} \right\|^{2+2\eta} |\mathcal{F}_{k} \right] \leq 2^{2+2\eta}C_{\eta}
\end{equation}
Then, applying Markov's inequality,
\begin{align*}
\mathbb{P}\left[ \sup_{0\leq k \leq n} \sqrt{\frac{n^{\alpha}}{c_{\alpha}}} \left\| \beta_{n+1,k+1} \gamma_{k+1}H^{-1} \xi_{k+1} \right\| > \nu \right] & \leq \sum_{k=0}^{n} \mathbb{P}\left[  \sqrt{\frac{n^{\alpha}}{c_{\alpha}}} \left\| \beta_{n+1,k+1} \gamma_{k+1}H^{-1} \xi_{k+1} \right\| > \nu \right] \\
& \leq \sum_{k=0}^{n} \mathbb{P}\left[ \left\| \xi_{k+1} \right\| > \frac{\sqrt{c_{\alpha}}\nu}{n^{\alpha/2}\left| \beta_{n+1,k+1}\right|\gamma_{k+1}\left\| H^{-1}\right\|_{op}} \right] \\
& \leq \frac{c_{\alpha}^{1+\eta}\left\| H^{-1} \right\|_{op}^{2+2\eta}}{\nu^{2+2\eta}} 2^{2+2\eta}C_{\eta}n^{\alpha\left( 1+ \eta \right)}\sum_{k=0}^{n}\left| \beta_{n+1,k+1} \right|^{2+2\eta} k^{-2\alpha\left( 1+ \eta \right)}.
\end{align*}
Moreover, note that there is a rank $n_{\alpha}$ such that for all $j \geq n_{\alpha}$ we have $c_{\alpha} j^{-\alpha} \leq 1$. For the sake of readibility of the proof (in other way, one can split the sum into two parts as in the proof of Lemma 3.1 in \cite{CCG2015}), we consider from now on that $n_{\alpha} =1$. Then
\begin{align*}
\left| \beta_{n+1,k+1} \right| & =  \prod_{i=k+2}^{n+1} \left( 1-c_{\alpha}i^{-\alpha} \right)  \leq  \exp \left( - c_{\alpha}\sum_{i=k+2}^{n+1} i^{-\alpha} \right).
\end{align*}
Applying Lemma E.2 in \cite{CG2015}, it comes
\[
 \sum_{k=0}^{n}\left| \beta_{n+1,k+1} \right|^{2+2\eta} k^{-2\alpha\left( 1+ \eta \right)} =\mathcal{O} \left( n^{-\alpha -2 \alpha \eta} \right) ,
\]
and so
\[
\lim_{n\to + \infty} \frac{c_{\alpha}^{ 1+ \eta}\left\| H^{-1} \right\|_{op}^{2+2\eta}}{\nu^{2 +2 \eta}} 2^{2+2\eta}C_{\eta}n^{\alpha (1+\eta)}\sum_{k=0}^{n}\left| \beta_{n+1,k+1} \right| k^{-2\alpha( 1+ \eta)} = 0 .
\]

\medskip

\noindent\textbf{Proof of equation (\ref{cond2}): } First, note that
\begin{align*}
\sum_{k=0}^{n} \beta_{n+1,k+1}^{2} \gamma_{k+1}^{2} H^{-1} \xi_{k+1}\xi_{k+1}^{T}H^{-1}& = \sum_{k=0}^{n} \beta_{n+1,k+1}^{2} \gamma_{k+1}^{2} H^{-1} \mathbb{E}\left[ \xi_{k+1}\xi_{k+1}^{T}|\mathcal{F}_{k} \right] H^{-1}\\ &+\sum_{k=0}^{n} \beta_{n+1,k+1}^{2} \gamma_{k+1}^{2} H^{-1} \Xi_{k+1}H^{-1}, 
\end{align*}
with $\Xi_{k+1} = \xi_{k+1}\xi_{k+1}^{T} - \mathbb{E}\left[ \xi_{k+1}\xi_{k+1}^{T} |\mathcal{F}_{k} \right]$. $\left( \Xi_{k} \right)$ is a sequence of martingale differences adapted to the filtration $\left( \mathcal{F}_{k} \right)$. Applying Theorem~\ref{theomartbeta}, it comes
\[
\left\| \sum_{k=0}^{n}\beta_{n+1,k+1}^{2}\gamma_{k+1}^{2}H^{-1} \Xi_{k+1} H^{-1} \right\|^{2} = O \left( \frac{\ln n}{n^{3\alpha}}\right) \quad a.s.
\]
Moreover, note that
\begin{align*}
\sum_{k=0}^{n} & \beta_{n+1,k+1}^{2} \gamma_{k+1}^{2} H^{-1} \mathbb{E}\left[ \xi_{k+1}\xi_{k+1}^{T}|\mathcal{F}_{k} \right] H^{-1} =  -\sum_{k=0}^{n} \beta_{n+1,k+1}^{2} \gamma_{k+1}^{2} H^{-1} \nabla G\left(\theta_{k} \right) \nabla G\left( \theta_{k} \right)^{T} H^{-1} \\
& + \sum_{k=0}^{n} \beta_{n+1,k+1}^{2} \gamma_{k+1}^{2} H^{-1} \mathbb{E}\left[ \nabla_{h} g \left( X_{k+1},Y_{k+1} , \theta_{k} \right) \nabla_{h} g \left( X_{k+1} , Y_{k+1} , \theta_{k}\right)^{T}|\mathcal{F}_{k} \right] H^{-1}.
\end{align*}
Applying Theorem \ref{ratetheta}, equality (\ref{vitexp}) and Lemma E.2 in \cite{CG2015}, since the gradient of $G$ is Lipschitz, one can check that for all $\delta > 0$,
\[
\left\| \sum_{k=0}^{n} \beta_{n+1,k+1}^{2} \gamma_{k+1}^{2} H^{-1} \nabla G\left(\theta_{k} \right) \nabla G\left( \theta_{k} \right)^{T} H^{-1} \right\|_{F}^{2} = o \left( \frac{(\ln n)^{1+\delta}}{n^{2\alpha}} \right) \quad a.s.
\]
Moreover, noting 
\[
R_{k}=\mathbb{E}\left[ \nabla_{h} g \left( X_{k+1},Y_{k+1} , \theta_{k} \right) \nabla_{h} g \left( X_{k+1} , Y_{k+1} , \theta_{k}\right)^{T}|\mathcal{F}_{k} \right] - \mathbb{E}\left[ \nabla_{h} g \left( X_{k+1},Y_{k+1} , \theta \right)\nabla_{h} g \left( X_{k+1},Y_{k+1} , \theta \right)^{T} |\mathcal{F}_{k} \right] ,
\]
\begin{align*}
& \sum_{k=0}^{n}  \beta_{n+1,k+1}^{2} \gamma_{k+1}^{2} H^{-1} \mathbb{E}\left[ \nabla_{h} g \left( X_{k+1},Y_{k+1} , \theta_{k} \right) \nabla_{h} g \left( X_{k+1} , Y_{k+1} , \theta_{k}\right)^{T}|\mathcal{F}_{k} \right] H^{-1} \\
&  = \sum_{k=0}^{n} \beta_{n+1,k+1}^{2} \gamma_{k+1}^{2} H^{-1} R_{k} H^{-1} + \sum_{k=1}^{n}  \beta_{n+1,k+1}^{2} \gamma_{k+1}^{2} H^{-1} \mathbb{E}\left[ \nabla_{h} g \left( X_{k+1},Y_{k+1} , \theta \right) \nabla_{h} g \left( X_{k+1} , Y_{k+1} , \theta\right)^{T}|\mathcal{F}_{k} \right] H^{-1}
\end{align*}
Moreover, applying Theorem \ref{ratetheta}, Lemma E.2 in \cite{CG2015}, one chan check that for all $\delta > 0$,
\[
\left\| \sum_{k=0}^{n} \beta_{n+1,k+1}^{2} \gamma_{k+1}^{2} H^{-1} R_{k} H^{-1} \right\|_{F} = o \left( \frac{(\ln n)^{1+\delta}}{n^{2\alpha}} \right) \quad a.s.
\]
Furthermore, since $\mathbb{E}\left[ \nabla_{h} g \left( X_{k+1},Y_{k+1} , \theta \right) \nabla_{h} g \left( X_{k+1} , Y_{k+1} , \theta\right)^{T}|\mathcal{F}_{k} \right] = \sigma^{2}H$, applying Lemma A.1 in \cite{GB2017}, 
\[
\lim_{n\to \infty}\left\| H^{-1} \left( \frac{n^{\alpha}}{c_{\alpha}}\sum_{k=0}^{n}  \beta_{n+1,k+1}^{2} \gamma_{k+1}^{2}  \mathbb{E}\left[ \nabla_{h} g \left( X_{k+1},Y_{k+1} , \theta \right) \nabla_{h} g \left( X_{k+1} , Y_{k+1} , \theta\right)^{T}|\mathcal{F}_{k} \right] - \frac{1}{2}\sigma^{2}H \right) H^{-1} \right\|_{F} =0,
\]
which concludes the proof.
\end{proof}

\subsection{Proof of Theorem \ref{ratethetabar}}
Let us first give some decompositions of the averaged estimates. First, note that
\[
\overline{\theta}_{n} = \frac{1}{n+1}\sum_{k=0}^{n} \theta_{k}.
\]
Second, one can write decomposition (\ref{decareutiliser}) as
\begin{equation}
\label{equationcle} \theta_{n} - \theta = H^{-1} \overline{S}_{n}\frac{\left( \theta_{n} - \theta \right) - \left( \theta_{n+1} - \theta \right)}{\gamma_{n+1}} + H^{-1}\xi_{n+1} - H^{-1} \delta_{n} .
\end{equation}
Then, summing these equalities and dividing by $n+1$, it comes
\begin{equation}
\label{decmoy} \overline{\theta}_{n} - \theta = H^{-1}\frac{1}{n+1}\sum_{k=0}^{n} \overline{S}_{k}\frac{\left( \theta_{k} - \theta \right) - \left( \theta_{k+1} - \theta \right)}{\gamma_{k+1}} + H^{-1} \frac{1}{n+1} \sum_{k=0}^{n} \xi_{k+1} - H^{-1} \frac{1}{n+1}\sum_{k=0}^{n} \delta_{k} .
\end{equation}

\begin{proof}[Proof of Theorem \ref{ratethetabar}]
Let us give the rate of convergence of each term on the right-hand side of equality (\ref{decmoy}). First, in order to apply a Law of Large Numbers and a Central Limit Theorem for martingales, let us calculate
\[
\lim_{n\to \infty}\frac{1}{n+1}\sum_{k=0}^{n} \mathbb{E}\left[ \xi_{k+1}\xi_{k+1}^{T} |\mathcal{F}_{k} \right].
\]
By definition of $\left( \xi_{n} \right)$, it comes 
\begin{align*}
\frac{1}{n+1}\sum_{k=0}^{n} \mathbb{E}\left[ \xi_{k+1}\xi_{k+1}^{T} |\mathcal{F}_{k} \right]&  = \frac{1}{n+1}\sum_{k=0}^{n} \mathbb{E}\left[ \nabla_{h} g \left( X_{k+1} ,Y_{k+1} , \theta_{k} \right)\nabla_{h} g \left( X_{k+1} ,Y_{k+1} , \theta_{k} \right)^{T} |\mathcal{F}_{k} \right] \\
&  - \frac{1}{n+1}\sum_{k=0}^{n} \nabla G\left( \theta_{k} \right) \nabla G \left( \theta_{k} \right)^{T} .  
\end{align*}
By continuity and since $\theta_{k}$ converges amost surely to $\theta$, Toeplitz lemma implies
\[
\frac{1}{n+1}\sum_{k=0}^{n} \mathbb{E}\left[ \xi_{k+1}\xi_{k+1}^{T} |\mathcal{F}_{k} \right] \xrightarrow[n\to + \infty]{a.s} \mathbb{E}\left[ \nabla_{h} g \left( X,Y , \theta \right)\nabla_{h} g \left( X,Y , \theta \right)^{T} \right] . 
\] 
Furthermore,
\[
\mathbb{E}\left[ \nabla_{h} g \left( X,Y , \theta \right)\nabla_{h} g \left( X,Y , \theta \right)^{T} \right] = \mathbb{E}\left[ \epsilon^{2} \nabla_{h}f\left( X,Y,\theta \right) \nabla_{h}f \left( X,Y , \theta \right)^{T} \right] = \sigma^{2} H .
\]
Finally, since $\mathbb{E}\left[ \left\| \xi_{n+1} \right\|^{2+2\eta} |\mathcal{F}_{n} \right] \leq 2^{2+2\eta}C_{\eta}$,  applying a Law of Large Numbers for martingales (see \cite{Duf97}),
\[
\frac{1}{(n+1)^{2}} \left\| H^{-1} \sum_{k=0}^{n} \xi_{k+1} \right\|^{2} = \mathcal{O}\left( \frac{\ln n}{n} \right) \quad a.s.
\]
Moreover, applying a Central Limit Theorem for martingales (see \cite{Duf97}),
\[
\frac{1}{\sqrt{n}}H^{-1} \sum_{k=0}^{n} \xi_{k+1} \xrightarrow[n\to + \infty]{\mathcal{L}} \mathcal{N}\left( 0 , \sigma^{2}H^{-1} \right) .
\]
Let us now prove that other terms on the right-hand side of equality (\ref{decmoy}) are negligible. Thanks to assumption \textbf{(H4b)} and since $\theta_{n}$ converges almost surely to $\theta$, 
\[
\left\| \delta_{n} \right\| \leq \left\| \theta_{n} - \theta \right\| \int_{0}^{1} \left\| \nabla^{2}G \left( \theta + t \left( \theta_{n} - \theta \right) \right) - H \right\|_{op} dt = O \left( \left\| \theta_{n} - \theta \right\|^{2} \right) \quad a.s
\] 
Then applying Theorem \ref{ratetheta}, for all $\delta > 0$,
\[
\frac{1}{n+1}\left\| H^{-1}\sum_{k=0}^{n} \delta_{k} \right\| \leq \frac{\left\| H^{-1} \right\|_{op}}{n+1} \sum_{k=1}^{n} \left\| \delta_{k} \right\| = o \left( \frac{(\ln n)^{1+\delta}}{n^{\alpha}} \right) \quad a.s 
\]
and since $\alpha > 1/2$, this term is negligible. Finally, as in \cite{Pel00}, using an Abel's transform,
\begin{align}
\notag \frac{1}{n+1}\sum_{k=0}^{n} \overline{S}_{k}\frac{\left( \theta_{k} - \theta \right) - \left( \theta_{k+1} - \theta \right)}{c_{\alpha}(k+1)^{-\alpha}} & = - \frac{\overline{S}_{n+1}\left( \theta_{n+1} - \theta\right) }{(n+1)\gamma_{n+2}} + \frac{\overline{S}_{0}\left( \theta_{0}-\theta \right)}{(n+1)\gamma_{1}} \\
\label{termedemerde}&  - \frac{1}{n+1}\sum_{k=1}^{n+1} \left( \gamma_{k}^{-1}\overline{S}_{k-1} - \gamma_{k+1}^{-1}\overline{S}_{k}  \right) \left( \theta_{k} - \theta \right) .
\end{align}
Let us now give the rates of convergence of each term on the right-hand side of equality (\ref{termedemerde}). First, applying Theorem \ref{ratetheta} and since $\overline{S}_{n}$ converges almost surely to $H$, 
\[
\left\| \frac{\overline{S}_{n+1}\left( \theta_{n+1} - \theta\right) }{\gamma_{n+2}(n+1)} \right\|^{2} = O \left( \frac{\ln n}{n^{2 - \alpha}} \right) \quad a.s \quad \quad \quad \text{and} \quad \quad \quad \left\| \frac{\overline{S}_{0}\left( \theta_{0}-\theta \right)}{(n+1)\gamma_{1}} \right\|^{2} = O \left( \frac{1}{n^{2}} \right) \quad a.s
\]
and these terms are negligible since $\alpha < 1$. Furthermore, since
\[
\overline{S}_{k}  =\overline{S}_{k-1} - \frac{1}{k}\overline{S}_{k-1} + \frac{1}{k} \left( \overline{\Phi}_{k} \overline{\Phi}_ {k}^{T} + \frac{c_{\beta}}{k^{\beta}} Z_{k}Z_{k}^{T} \right),
\]
it comes
\begin{align*}
 (*) :& =\frac{1}{n+1}\sum_{k=1}^{n+1} \left( \gamma_{k}\overline{S}_{k-1} - \gamma_{k+1}\overline{S}_{k}  \right) \left( \theta_{k} - \theta \right) \\
& =  \frac{1}{n+1}\sum_{k=1}^{n+1} \left( \gamma_{k}^{-1} - \gamma_{k+1}^{-1} \right) \overline{S}_{k-1} \left( \theta_{k} - \theta \right) - \frac{1}{n+1}\sum_{k=1}^{n+1}\frac{\gamma_{k+1}^{-1}}{k}\overline{S}_{k-1}\left( \theta_{k} - \theta \right)  \\
&  + \underbrace{\frac{1}{n+1}\sum_{k=1}^{n+1}\frac{\gamma_{k+1}^{-1}}{k } \left(   \overline{\Phi}_{k} \overline{\Phi}_{k}^{T} + \frac{c_{\beta}}{k^{\beta}} Z_{k}Z_{k}^{T} \right) \left( \theta_{k} - \theta \right)}_{:=(**)} .
\end{align*}
For the first term on the right-hand side of previous equality, since $\left| \gamma_{k}^{-1} - \gamma_{k+1}^{-1} \right| \leq  \alpha c_{\alpha}^{-1} k^{\alpha -1}$ and since $\overline{S}_{k-1}$ converges almost surely to $H$, applying Theorem \ref{ratetheta}, for all $\delta > 0$,
\[
\frac{1}{(n+1)^{2}}\left\| \sum_{k=1}^{n+1} \left( \gamma_{k}^{-1} - \gamma_{k+1}^{-1} \right) \overline{S}_{k-1} \left( \theta_{k} - \theta \right) \right\|^{2} = o \left( \frac{(\ln n)^{1+\delta}}{n^{2-\alpha} }\right) \quad a.s
\]
which is negligible since $\alpha < 1$. Furthermore, since $\overline{S}_{k-1}$ converges almost surely to $H$, one can check that
\[
\frac{1}{(n+1)^{2}} \left\|  \sum_{k=1}^{n} \frac{\gamma_{k+1}^{-1}}{k} \overline{S}_{k-1}\left( \theta_{k} - \theta \right) \right\|^{2} = o \left( \frac{(\ln n)^{1+\delta}}{n^{2- \alpha}} \right) \quad a.s.
\]
Let us now give the rate of convergence of $(**)$. In this aim, let us consider $\delta > 0$ and introduce the events $\Omega_{k}= \left\lbrace \left\| \theta_{k}   - \theta \right\| \leq \frac{(\ln k)^{1/2 + \delta}}{k^{\alpha/2 }} \right\rbrace$. Since $\delta > 0$, and thanks to Theorem \ref{ratetheta}, $\frac{\left\| \theta_{n} - \theta \right\|^{2}n^{\alpha}}{(\ln n)^{1+\delta}}$ converges almost surely to $0$, so that $\mathbf{1}_{\Omega_{k}^{C}}$ converges almost surely to $0$. Furthermore,
\begin{align*}
(**) &= \underbrace{\frac{1}{n+1}\sum_{k=1}^{n+1}\frac{\gamma_{k+1}^{-1}}{k } \left(   \overline{\Phi}_{k} \overline{\Phi}_{k}^{T} + \frac{c_{\beta}}{k^{\beta}} Z_{k}Z_{k}^{T} \right) \left( \theta_{k} - \theta \right)\mathbf{1}_{\Omega_{k}}}_{:= (***)} \\&+ \frac{1}{n+1}\sum_{k=1}^{n+1}\frac{\gamma_{k+1}^{-1}}{k } \left(   \overline{\Phi}_{k} \overline{\Phi}_{k}^{T} + \frac{c_{\beta}}{k^{\beta}} Z_{k}Z_{k}^{T} \right) \left( \theta_{k} - \theta \right) \mathbf{1}_{\Omega_{k}^{C}}
\end{align*}
Since $\mathbf{1}_{\Omega_{k}^{C}}$ converges almost surely to $0$,
\[
\sum_{k\geq 1}\frac{\gamma_{k+1}^{-1}}{k } \left\|   \overline{\Phi}_{k} \overline{\Phi}_{k}^{T} + \frac{c_{\beta}}{k^{\beta}} Z_{k}Z_{k}^{T} \right\| \left\| \theta_{k} - \theta \right\|_{op} \mathbf{1}_{\Omega_{k}^{C}} < + \infty \quad a.s
\]
so that
\[
\frac{1}{n+1}\left\| \sum_{k=1}^{n+1}\frac{\gamma_{k+1}^{-1}}{k } \left(   \overline{\Phi}_{k} \overline{\Phi}_{k}^{T} + \frac{c_{\beta}}{k^{\beta}} Z_{k}Z_{k}^{T} \right) \left( \theta_{k} - \theta \right) \mathbf{1}_{\Omega_{k}^{C}} \right\| = O \left( \frac{1}{n}\right) \quad a.s.
\]
Moreover,
\begin{align*}
(***) \leq \frac{1}{n+1}\sum_{k=1}^{n+1} \frac{\gamma_{k+1}^{-1}}{k}\left\| \overline{\Phi}_{k} \overline{\Phi}_{k}^{T} + \frac{c_{\beta}}{k^{\beta}}Z_{k}Z_{k}^{T} \right\|_{op} \frac{(\ln k)^{1/2 + \delta}}{k^{\alpha /2}} .
\end{align*}
One can consider the sequence of martigales differences $\left( \Xi_{k} \right)$ adapted to the filtration $\left( \mathcal{F}_{k} \right)$ and defined for all $k \geq 1$ by
\[
\Xi_{k} = \left\| \overline{\Phi}_{k} \overline{\Phi}_{k}^{T} + \frac{c_{\beta}}{k^{\beta}}Z_{k}Z_{k}^{T} \right\|_{op} - \mathbb{E}\left[ \left\| \overline{\Phi}_{k} \overline{\Phi}_{k}^{T} + \frac{c_{\beta}}{k^{\beta}}Z_{k}Z_{k}^{T} \right\|_{op}  |\mathcal{F}_{k-1} \right] .
\]
Then, 
\[
(***) \leq \frac{1}{n+1}\sum_{k=1}^{n+1} \frac{\gamma_{k+1}(\ln k)^{1/2 + \delta}}{k^{\alpha/2 +1}} \mathbb{E}\left[ \left\| \overline{\Phi}_{k} \overline{\Phi}_{k}^{T} + \frac{c_{\beta}}{k^{\beta}}Z_{k}Z_{k}^{T} \right\|_{op} |\mathcal{F}_{k-1} \right] + \frac{1}{n+1}\sum_{k=1}^{n+1} \frac{\gamma_{k+1}(\ln k)^{1/2 + \delta}}{k^{\alpha/2 +1}} \Xi_{k} .
\]
\end{proof}
Then, thanks to Assumption \textbf{(H1b)} and Toeplitz lemma, 
\[
\frac{1}{n+1}\sum_{k=1}^{n+1} \frac{\gamma_{k+1}(\ln k)^{1/2 + \delta}}{k^{\alpha/2 +1}} \mathbb{E}\left[ \left\| \overline{\Phi}_{k} \overline{\Phi}_{k}^{T} + \frac{c_{\beta}}{k^{\beta}}Z_{k}Z_{k}^{T} \right\|_{op} |\mathcal{F}_{k-1} \right] = o \left( \frac{(\ln n)^{1/2 + \delta}}{n^{1- \alpha/2}} \right) \quad a.s .
\]
Furthermore, since $\alpha < 1$, with assumption \textbf{(H1b)}, Theorem~\ref{theomartmoy} leads  to
\[
\left\| \frac{1}{n+1}\sum_{k=1}^{n+1} \frac{\gamma_{k+1}(\ln k)^{1/2 + \delta}}{k^{\alpha/2 +1}} \Xi_{k} \right\|^{2} = O \left( \frac{1}{n^{2}} \right) \quad a.s.
\]

\subsection{Proof of Corollary \ref{ratesn}}
\begin{proof}
The aim is to give the rate of convergence of each term of decomposition (\ref{decsn}). Note that the rate of the martingale term is given by (\ref{vitXi}), while, thanks to equation (\ref{vitzi}), we have
\[
\frac{1}{n^{2}}\left\| \sum_{k=1}^{n} \frac{c_{\beta}}{k^{\beta}} Z_{k+1}Z_{k+1}^{T} \right\|_{F}^{2} = O \left( \frac{1}{n^{2\beta}} \right) \quad a.s.
\]
Finally, since the functional $h \longmapsto \mathbb{E}\left[ \nabla_{h}f \left( X , h \right) \nabla_{h} f ( X,h )^{T} \right]$ is $C_{f}$-Lipschitz on a neighborhood of $\theta$ and since $\overline{\theta}_{n}$ converges almost surely to $\theta$, applying Theorem \ref{ratethetabar} as well as Toeplitz lemma, for all $\delta > 0$,
\begin{align*}
\frac{1}{n^{2}}\left\| \sum_{k=1}^{n} \left( \mathbb{E}\left[ \overline{\Phi}_{k+1}\overline{\Phi}_{k+1}^{T} |\mathcal{F}_{k} \right] - H \right) \right\|_{F}^{2} & = \mathcal{O} \left(  \frac{C_{f}^{2}}{n^{2}} \left( \sum_{k=1}^{n} \left\| \overline{\theta}_{k} - \theta \right\| \right)^{2} \right) \quad a.s  \\
& = o \left( \frac{(\ln n)^{1+\delta}}{n} \right) \quad a.s,
\end{align*}
which concludes the proof.

\begin{rmq}\label{rempourletrucpourri}
Note that to prove equality \eqref{tlc} without knowing the rate of convergence of $\overline{\theta}_{n}$, it is necessary to have a first rate of convergence of $\overline{S}_{n}$. For that purpose, we study the asymptotic behaviour of
\[
\frac{1}{n^{2}}\left\| \sum_{k=1}^{n} \left( \mathbb{E}\left[ \overline{\Phi}_{k+1}\overline{\Phi}_{k+1}^{T} |\mathcal{F}_{k} \right] - H \right) \right\|_{F}^{2}.
\]
Equality \eqref{ratethetaas} yields for all $\delta > 0$,
\[
\left\| \overline{\theta}_{n} - \theta \right\| \leq \frac{1}{n+1}\sum_{k=0}^{n} \left\| \theta_{k} - \theta \right\| = O \left( \frac{(\ln n)^{1 /2 + \delta /2}}{n^{\alpha /2}} \right) , 
\]
so that, since the functional $h \longmapsto \mathbb{E}\left[ \nabla_{h}f \left( X , h \right) \nabla_{h} f ( X,h )^{T} \right]$ is $C_{f}$-Lipschitz on a neighborhood of $\theta$,
\begin{align*}
\frac{1}{n^{2}}\left\| \sum_{k=1}^{n} \left( \mathbb{E}\left[ \overline{\Phi}_{k+1}\overline{\Phi}_{k+1}^{T} |\mathcal{F}_{k} \right] - H \right) \right\|_{F}^{2} & = \mathcal{O} \left(  \frac{C_{f}^{2}}{n^{2}} \left( \sum_{k=1}^{n} \left\| \overline{\theta}_{k} - \theta \right\| \right)^{2} \right) \quad a.s  \\
& = o \left( \frac{(\ln n)^{1+\delta}}{n^{\alpha}} \right) \quad a.s,
\end{align*}
and then, since $\beta < \alpha  -1/2$,
\[
\left\| \overline{S}_{n} - H \right\|_{F} = O  \left( \max \left\lbrace  \frac{c_{\beta}}{n^{\beta}} , \sqrt{\frac{(\ln n)^{1+\delta}}{n^{\alpha}}} \right\rbrace  \right) \quad a.s
\]
\end{rmq}
\end{proof}

\subsection{Proof of Theorem \ref{theothetatilde}\label{prooftilde}}
Only the main lines of the proof are given since it is a mix between the proof of Theorem~\ref{ratethetabar} and the ones in \cite{BGBP2019}.

\begin{proof}[Proof of THeorem \ref{theothetatilde}]
Let us denote $\overline{H}_{n}^{-1} = (n+1) \widetilde{H}_{n}^{-1}$. First, remark that as in the proof of Corollary \ref{ratesn}, one can check that that on $\widetilde{\Gamma}_{\theta}$,
\[
\overline{H}_{n} \xrightarrow[n\to + \infty]{a.s} H \quad \quad \text{and} \quad \quad \overline{H}_{n}^{-1} \xrightarrow[n\to + \infty]{a.s} H^{-1} .
\]
Furthermore, decomposition \eqref{decxi} can be rewritten as
\begin{equation}
\label{decxitilde} \widetilde{\theta}_{n+1} = \widetilde{\theta}_{n} - \theta - \frac{1}{n+1} \overline{S}_{n}^{-1} \nabla G \left( \widetilde{\theta}_{n} \right) + \frac{1}{n+1}\overline{H}_{n}^{-1} \widetilde{\xi}_{n+1},
\end{equation}
where $\widetilde{\xi}_{n+1} := \nabla G \left( \widetilde{\theta}_{n} \right) - \nabla_{h} g \left( X_{n+1} , Y_{n+1} , \widetilde{\theta}_{n} \right)$. Then, $\left( \widetilde{\xi}_{n} \right)$ is a sequence of martingale differences adapted to the filtration $\left( \mathcal{F}_{n} \right)$. Linearizing the gradient, decomposition \eqref{decdelta} can be rewritten as
\begin{align}
\notag \widetilde{\theta}_{n+1} - \theta & = \frac{n}{n+1} \left( \widetilde{\theta}_{n} - \theta \right) + \frac{1}{n+1}\left( H^{-1} - \overline{H}_{n}^{-1} \right)H \left( \widetilde{\theta}_{n} - \theta \right) \\
\label{decdeltatilde} & + \frac{1}{n+1}\overline{H}_{n}^{-1}\widetilde{\xi}_{n+1} - \frac{1}{n+1}\overline{H}_{n}^{-1} \widetilde{\delta}_{n} 
\end{align}
where $\widetilde{\delta}_{n} = \nabla G \left( \widetilde{\theta}_{n} \right) - H \left( \widetilde{\theta}_{n} - \theta \right)$ is the reminder term in the Taylor's decomposition of the gradient. Then, by induction, it comes
\begin{align}
\label{dectot} \widetilde{\theta}_{n} - \theta = \frac{1}{n}\left( \widetilde{\theta}_{0} - \theta \right)+ \frac{1}{n}\sum_{k=0}^{n-1} \left( H^{-1} - \overline{H}_{k}^{-1} \right) H \left( \widetilde{\theta}_{k} - \theta \right) - \frac{1}{n}\sum_{k=0}^{n-1} \overline{H}_{k}^{-1} \widetilde{\delta}_{k} + \frac{1}{n}\sum_{k=0}^{n-1}\overline{H}_{k}^{-1}\widetilde{\xi}_{k+1} .
\end{align}
Since $\overline{H}_{k}^{-1}$ converges almost surely to $H^{-1}$, on $\widetilde{\Gamma}_{\theta}$, one can check that 
\[
\frac{1}{n}\sum_{k=0}^{n-1} \overline{H}_{k}^{-1} \widetilde{\xi}_{k+1}\widetilde{\xi}_{k+1}^{T}\overline{H}_{k}^{-1} \xrightarrow[n\to + \infty]{a.s} \sigma^{2}H^{-1}.
\]
Thanks to assumption \textbf{(H5)}, applying a law of large numbers for martingales, one can check that
\begin{equation}\label{vitxitilde}
\left\|\frac{1}{n}\sum_{k=0}^{n-1} \overline{H}_{k}^{-1} \widetilde{\xi}_{k+1} \right\|^{2} = \mathcal{O} \left( \frac{\ln n}{n} \right) \quad a.s
\end{equation}
In the same way, applying  a central limit theorem, it comes
\begin{equation}\label{tlctilde}
\frac{1}{\sqrt{n}} \sum_{k=0}^{n-1} \overline{H}_{k}^{-1} \widetilde{\xi}_{k+1} \xrightarrow[n\to + \infty]{\mathcal{L}} \mathcal{N}\left( 0 , \sigma^{2}H^{-1} \right) .
\end{equation}
Let us now prove that the other terms in decomposition \eqref{dectot} are negligible. First, clearly 
\begin{equation}
\label{vitdebile}\frac{1}{n}\left\| \widetilde{\theta}_{0} - \theta \right\| = \mathcal{O} \left( \frac{1}{n} \right) \quad a.s.
\end{equation}
Furthermore, let us denote 
\[
\widetilde{\Delta}_{n} = \frac{1}{n}\sum_{k=0}^{n-1} \left( H^{-1} - \overline{H}_{k}^{-1} \right) H \left( \widetilde{\theta}_{k} - \theta \right) - \frac{1}{n}\sum_{k=0}^{n-1} \overline{H}_{k}^{-1} \widetilde{\delta}_{k} .
\]
Furthermore, as in the proof of Theorem \ref{ratetheta}, one can verify
\begin{align*}
& \left\| \left( H^{-1} - \overline{H}_{n}^{-1} \right) H \left( \widetilde{\theta}_{n} - \theta \right) \right\| = o \left( \left\| \widetilde{\theta}_{n} - \theta \right\| \right) \quad a.s
& \left\| \widetilde{\delta}_{n} \right\| = o \left( \left\| \widetilde{\theta}_{n} - \theta \right\| \right) \quad a.s
\end{align*}
Then,
\begin{equation}\label{deltatilde}
\left\| \widetilde{\Delta}_{n+1} \right\| = \frac{n}{n+1} \left\| \widetilde{\Delta}_{n} \right\| + o \left( \widetilde{\theta}_{n} - \theta \right) \quad a.s.
\end{equation}
As in the proof of Theorem 6.2 in \cite{BGBP2019} (see equations (6.24) to (6.32)), it comes
\begin{equation}
\label{vitdeltatilde}\left\| \widetilde{\Delta}_{n} \right\|^{2} = \mathcal{O} \left( \frac{\ln n}{n} \right) \quad a.s.
\end{equation}
Then, thanks to equalities \eqref{vitxitilde},\eqref{vitdebile} and \eqref{vitdeltatilde}, it comes
\begin{equation}\label{vitthetatilde}
\left\|\widetilde{\theta}_{n} - \theta \right\|^{2} = \mathcal{O} \left( \frac{\ln n}{n} \right) \quad a.s.
\end{equation}
In order to get the asymptotic normality of $\widetilde{\theta}_{n}$, let us now give the rate of convergence of each term on the right-hand side of decomposition \eqref{dectot}. First, since
\[
\widetilde{\delta}_{n} = \mathcal{O} \left( \left\| \widetilde{\theta}_{n} - \theta \right\|^{2} \right) \quad a.s,
\]
and since $\overline{H}_{n}^{-1}$ converges almost surely to $H^{-1}$, thanks to equality \eqref{vitthetatilde}, for all $\delta > 0$ 
\[
\frac{1}{n}\left\| \sum_{k=0}^{n-1} \overline{H}_{k}^{-1}\widetilde{\delta}_{k} \right\| = o \left( \frac{(\ln n)^{2+\delta}}{n} \right) \quad a.s
\]
which is so negligible. Furthermore, as in the proof of Corollary \ref{ratesn}, one can check that for all $\delta > 0$,
\[
\left\| \overline{H}_{n}^{-1} - H^{-1} \right\|^{2} = \mathcal{O} \left( \max \left\lbrace \frac{c_{\beta}^{2}}{n^{2\beta}} , \frac{(\ln n)^{1+\delta}}{n} \right\rbrace \right) \quad a.s .
\]
Then, for all $\delta > 0$,
\[
\frac{1}{n}\left\|\sum_{k=1}^{n} \left( \overline{H}_{k}^{-1} - H^{-1} \right) H \left( \widetilde{\theta}_{k} - \theta \right) \right\| = o \left( \max \left\lbrace \frac{c_{\beta}(\ln n)^{1/2+\delta}}{n^{\beta + 1/2}} , \frac{(\ln n)^{2+\delta}}{n} \right\rbrace \right) \quad a.s
\]
and this term is no negligible.
\end{proof} 

\subsection{Proof of Corollary \ref{estsigma}}
First, note that
\[
\left( \hat{Y}_{k} - Y_{k} \right)^{2} - \sigma^{2} = \left( f \left( X_{k} , \overline{\theta}_{k-1} \right) - f \left( X_{k} , \theta \right) \right)^{2} -2 \epsilon_{k} \left( f \left( X_{k} , \overline{\theta}_{k-1} \right) - f \left( X_{k} , \theta \right) \right) + \left( \epsilon_{k}^{2} - \sigma^{2} \right) .
\]
Thanks to a Taylor's decomposition, there are $U_{0},\ldots,U_{n-1} \in \mathbb{R}^{q}$ such that 
\[
R_{1} := \frac{1}{n}\sum_{k=1}^{n} \left( f \left( X_{k} , \overline{\theta}_{k-1} \right) - f \left( X_{k} , \theta \right) \right)^{2} \leq \frac{1}{n}\sum_{k=1}^{n} \left\| \nabla_{h} f \left( X_{k} , U_{k-1} \right) \right\|^{2} \left\| \overline{\theta}_{k-1} - \theta \right\|^{2}
\]
Let us consider the filtration $\left( \mathcal{F}^{(1)}_{n} \right)$ defined by $\mathcal{F}_{n}^{(1)} = \sigma \left( \left( X_{k} , \epsilon_{k-1}\right), k=1,\ldots ,n  \right)$. Then, considering $\widetilde{\xi}_{k} = \mathbb{E}\left[ \left\| \nabla_{h} f \left( X_{k} , U_{k-1} \right) \right\|^{2} |\mathcal{F}^{(1)}_{k-1} \right] - \left\| \nabla_{h} f \left( X_{k} , U_{k-1} \right) \right\|^{2}$, it comes thanks to assumption \textbf{(H1b)},
\[
R_{1} \leq \frac{1}{n}\sum_{k=1}^{n}\sqrt{C''} \left\| \overline{\theta}_{k-1} - \theta \right\|^{2} -  \frac{1}{n}\sum_{k=1}^{n} \widetilde{\xi}_{k} \left\| \overline{\theta}_{k-1} - \theta \right\|^{2}  
\]
Then, applying Toeplitz Lemma and Theorem \ref{ratethetabar} to the term on the left hand-side of previous inequality, one can check that for all $\delta > 0$,
\[
\frac{1}{n}\sum_{k=1}^{n}\sqrt{C''} \left\| \overline{\theta}_{k-1} - \theta \right\|^{2}  = o \left( \frac{(\ln n)^{2+ \delta}}{n}\right) \quad a.s.
\]
Furthermore, since $\left( \widetilde{\xi}_{n} \right)$ is a sequence of martingale differences with bounded squared moments, with the help of Theorems \ref{ratethetabar} and \ref{theomartmoy}, it comes
\[
\frac{1}{n}\sum_{k=1}^{n} \widetilde{\xi}_{k} \left\| \overline{\theta}_{k-1} - \theta \right\|^{2} = \mathcal{O} \left( \frac{1}{n} \right) \quad a.s
\]
so that
\[
R_{1} = o \left( \frac{(\ln n)^{2+\delta}}{n}\right) \quad a.s.
\] 
Furthermore, let us consider
\[
R_{2} = \frac{1}{n}\sum_{k=1}^{n}\epsilon_{k} \left( f \left( X_{k} , \overline{\theta}_{k-1} \right) - f \left( X_{k} , \theta \right) \right) .
\]
We have a sum of martingale differences, and thanks to Theorems \ref{ratethetabar} and \ref{theomartmoy}, one can check that
\[
\left\| R_{2} \right\|^{2} = o \left( \frac{(\ln n)^{2+\delta}}{n^{2}} \right) \quad a.s.
\]
Then, considering that
\[
\epsilon_{n}^{2} - \sigma^{2} = \epsilon_{n}^{2} - \mathbb{E}\left[ \epsilon_{n}^{2} |\mathcal{F}^{(1)}_{n-1} \right]
\]
one can apply a Central Limit Theorem for martingale to get the asymptotic normality, and a strong law of large numbers for martingales to get the almost sure rate of convergence.

\section{Useful results on martingales}\label{sectionmartingales}

\subsection{A useful theorem for stochastic algorithms with step sequence $(n^{-\alpha})_n$}

\begin{theo}\label{theomartbeta}
Let $H$ be a separable Hilbert space and let us consider
\[
M_{n+1} = \sum_{k=1}^{n} \beta_{n,k} \gamma_{k}  R_{k} \xi_{k+1},
\]
where
\begin{itemize}
\item $\left( \xi_{n} \right)$ is a $H$-valued martingale differences  sequence adapted  to a filtration $\left( \mathcal{F}_{n} \right)$ such that \begin{align}
\notag & \mathbb{E}\left[ \left\| \xi_{n+1} \right\|^{2} |\mathcal{F}_{n} \right] \leq C + R_{2,n} \quad a.s, \\
\label{hypmart} & \sum_{n\geq 1}\gamma_{n}\mathbb{E}\left[ \left\| \xi_{n+1} \right\|^{2}\ind{\left\| \xi_{n+1} \right\|^{2} \geq \gamma_{n}^{-1}(\ln n)^{-1}} |\mathcal{F}_{n} \right] < + \infty \quad a.s, 
\end{align}
where $C\geq 0$ and $(R_{2,n})_n$ converges almost surely to $0$;
\item $\gamma_{n}=cn^{-\alpha}$ with $c> 0$ and $\alpha \in (1/2,1)$;
\item $(R_{n})$ is a sequence of operators on $H$ such that, for a deterministic sequence $\left( v_{n} \right)$, 
\[
\left\| R_{n} \right\|_{op} = o \left( v_{n} \right) \quad a.s \quad \quad \text{and} \quad \quad v_{n} = \frac{(\ln n)^{a}}{n^{b}} .
\]
with $a,b \geq 0$;
\item For all $n \geq 1$ and $1 \leq k \leq n$,
\[
\beta_{n,k} = \prod_{j=k+1}^{n} \left( I_{H} - \gamma_{k} \Gamma \right)\quad \text{and} \quad \beta_{n,n} = I_{H},
\]
where $\Gamma $ is a symmetric operator on $H$ such that $0 < \lambda_{\min} (\Gamma ) \leq \lambda_{\max} (\Gamma) < + \infty $.
\end{itemize}
Then, 
\[
\left\| M_{n+1} \right\|^{2} = O \left(  \gamma_{n}v_{n}^{2}\ln n \right)  \quad a.s.
\]
\end{theo}

\begin{rmq}
Note equation \eqref{hypmart} holds since there are $\eta > \frac{1}{\alpha} -1$ and a positive constant $C_{2}$ such that
\[
\mathbb{E}\left[ \left\| \xi_{n+1} \right\|^{2+2\eta} |\mathcal{F}_{n} \right] \leq C_{2} + R_{\eta,n}
\]
with $R_{\eta,n}$ converging to $0$.
\end{rmq}

\begin{rmq}
Previous theorem remains true considering a sequence $(R_{n})$ satisfying that there are a positive constant $C_{R}$ and a rank $n_{R}$ such that for all $n \geq n_{R}$, $\left\| R_{n} \right\|_{op} \leq C_{R}v_{n}$.
\end{rmq}

\begin{proof}
Let us now consider the events
\begin{align*}
& A_{n} = \left\lbrace R_{n} > v_{n} \quad \text{or} \quad R_{2,n}> C  \right\rbrace \\
& B_{n+1} =  \left\lbrace R_{n} \leq v_{n}, R_{2,n} \leq C, \left\| \xi_{n+1} \right\|\leq \delta_{n}  \right\rbrace \\
& C_{n+1} = \left\lbrace R_{n} \leq v_{n}, R_{2,n} \leq C, \left\| \xi_{n+1} \right\| > \delta_{n}  \right\rbrace
\end{align*}
with $\delta_{n} = \gamma_{n}^{-1/2}(\ln n)^{-1/2}$. One can remark that $A_{n}^{c} = B_{n+1} \sqcup C_{n+1}$. Then, one can write $M_{n+1}$ as
\begin{align*}
M_{n+1} & = \sum_{k=1}^{n} \beta_{n,k} \gamma_{k}  R_{k} \xi_{k+1}\ind{A_{k}} + \sum_{k=1}^{n} \beta_{n,k} \gamma_{k}  R_{k} \xi_{k+1}\ind{A_{k}^{c}} \\
& = \sum_{k=1}^{n} \beta_{n,k} \gamma_{k}  R_{k} \xi_{k+1}\ind{A_{k}} + \sum_{k=1}^{n} \beta_{n,k} \gamma_{k}  R_{k} \left( \xi_{k+1}\ind{B_{k+1}} - \mathbb{E}\left[ \xi_{k+1} \ind{B_{k+1}} |\mathcal{F}_{k} \right] \right) \\
& + \sum_{k=1}^{n} \beta_{n,k} \gamma_{k}  R_{k} \left( \xi_{k+1}\ind{C_{k+1}} - \mathbb{E}\left[ \xi_{k+1} \ind{C_{k+1}} |\mathcal{F}_{k} \right] \right).
\end{align*}
Let us now give the rates of convergence of these three terms.

\bigskip

\noindent\textbf{Bounding $ M_{1,n+1} := \sum_{k=1}^{n} \beta_{n,k} \gamma_{k}  R_{k} \xi_{k+1}\ind{A_{k}} $. }
Remark that there is a rank $n_{0}$ such that for all $n \geq n_{0}$, $\left\| I_{H} - \gamma_{n} \Gamma \right\|_{op} \leq \left( 1- \lambda_{\min} \gamma_{n} \right) $. Furthermore, $M_{1,n+1} = \left( I_{H} - \gamma_{n} \Gamma \right) M_{1,n} + \gamma_{n}  R_{n} \xi_{n+1} \ind{A_{n}}$. Then, for all $n \geq n_{0}$,
\[
\mathbb{E}\left[ \left\| M_{1,n+1} \right\|^{2} |\mathcal{F}_{n} \right] \leq \left( 1- \lambda_{\min}\gamma_{n} \right)^{2} \left\| M_{1,n} \right\|^{2} + \gamma_{n}^{2} \left\| R_{n} \right\|_{op}^{2} \left( C + R_{2,n} \right) \ind{A_{n}} .
\]
Considering $V_{n+1} = \prod_{k=1}^{n} \left( 1 +  \lambda_{\min} \gamma_{k} \right)^{2} \left\| M_{1,n+1} \right\|^{2}$, it comes
\[
\mathbb{E}\left[ V_{n+1} |\mathcal{F}_{n} \right] \leq \left( 1 - \lambda_{\min}^{2} \gamma_{n}^{2} \right)^{2} V_{n} + \prod_{k=1}^{n} \left( 1 +  \lambda_{\min} \gamma_{k} \right)^{2}\gamma_{n}^{2}  \left\| R_{n} \right\|_{op}^{2} \left( C + R_{2,n} \right) \ind{A_{n}} 
\] 
Moreover, $\ind{A_{n}}$ converges almost surely to $0$ so that
\[
\sum_{n\geq 1} \prod_{k=1}^{n} \left( 1 +  \lambda_{\min} \gamma_{k} \right)^{2}\gamma_{n}^{2}  \left\| R_{n} \right\|_{op}^{2} \left( C + R_{2,n} \right) \ind{A_{n}} < + \infty \quad a.s
\]
and applying Robbins-Siegmund Theorem, $V_{n}$ converges almost surely to a finite random variable, i.e
\[
\left\| M_{1,n+1} \right\|^{2} = \mathcal{O} \left( \prod_{k=1}^{n} \left( 1 +  \lambda_{\min} \gamma_{k} \right)^{-2}  \right) \quad a.s
\]
and converges exponentially fast.

\bigskip

\noindent\textbf{Bounding $M_{2,n+1} := \sum_{k=1}^{n} \beta_{n,k} \gamma_{k}  R_{k} \left( \xi_{k+1}\ind{B_{k+1}} - \mathbb{E}\left[ \xi_{k+1} \ind{B_{k+1}} |\mathcal{F}_{k} \right] \right)$. }
Let us denote $\Xi_{k+1} =   R_{k} \left( \xi_{k+1}\ind{B_{k+1}} - \mathbb{E}\left[ \xi_{k+1} \ind{B_{k+1}} |\mathcal{F}_{k} \right] \right)$. Remark that $\left( \Xi_{n} \right)$ is a sequence of martingale differences adapted to the filtration $\left( \mathcal{F}_{n} \right)$. As in \cite{Pinelis} (proofs of Theorems~3.1 and 3.2), let $\lambda > 0$ and consider for all $t \in [ 0, 1]$ and $j \leq n$,
\[
\varphi(t) = \mathbb{E}\left[ \cosh \left( \lambda \left\| \sum_{k=1}^{j-1} \beta_{n,k}\gamma_{k} \Xi_{k+1} + t \beta_{n,j} \gamma_{j}\Xi_{j+1} \right\| \right) \left|\mathcal{F}_{j}\right. \right].
\]
One can check that $\varphi '(0) = 0$ and (see Pinellis for more details)
\[
\varphi '' (t) \leq \lambda^{2} \mathbb{E}\left[ \left\| \beta_{n,j} \gamma_{j}\Xi_{j+1} \right\|^{2}e^{ \lambda t \left\| \beta_{n,j} \gamma_{j}\Xi_{j+1} \right\|} \cosh \left( \lambda \left\| \sum_{k=1}^{j-1} \beta_{n,k}\gamma_{k} \Xi_{k+1} \right\| \right) |\mathcal{F}_{j} \right]
\]
Then,
\begin{align*}
\mathbb{E}\left[ \cosh \left( \lambda \left\| \sum_{k=1}^{j} \beta_{n,k}\gamma_{k} \Xi_{k+1} \right\| \right) | \mathcal{F}_{j} \right]	  = \varphi (1) & = \varphi (0) + \int_{0}^{1} (1-t) \varphi ''(t) dt \\
& \leq \left( 1+ e_{j,n } \right)\cosh \left( \lambda \left\| \sum_{k=1}^{j-1} \beta_{n,k}\gamma_{k} \Xi_{k+1} \right\| \right) 
\end{align*}
with $e_{j,n} = \mathbb{E}\left[ e^{\lambda \left\| \beta_{n,j} \gamma_{j}\Xi_{j+1} \right\|} -1 - \lambda\left\| \beta_{n,j} \gamma_{j}\Xi_{j+1} \right\| |\mathcal{F}_{j} \right]$, which is well defined since $\Xi_{j+1}$ is a.s. finite. Moreover, considering
\[
G_{n+1} = \frac{\cosh \left( \lambda \left\| \sum_{k=1}^{n} \beta_{n,k}\gamma_{k} \Xi_{k+1} \right\| \right)}{\prod_{j=1}^{n} \left( 1 + e_{j,n} \right)} \quad \quad \text{and} \quad \quad G_{0} = 1
\]
and since $\mathbb{E}\left[ G_{n+1} |\mathcal{F}_{n} \right] =G_{n}$, it comes $\mathbb{E}\left[ G_{n+1} \right] = 1 $. For all $r > 0$,
\begin{align*}
\mathbb{P}\left[ \left\| M_{2,n+1} \right\| \geq r \right] & = \mathbb{P}\left[ G_{n+1} \geq \frac{\cosh (\lambda r )}{\prod_{j=1}^{n}\left( 1+ e_{j,n} \right)} \right]  \leq \mathbb{P}\left[ 2G_{n+1} \geq \frac{e^{\lambda r }}{\prod_{j=1}^{n} \left( 1+ e_{jn} \right)} \right].
\end{align*}
Furthermore, let $\epsilon_{j+1} = \xi_{j+1}\ind{B_{j}} - \mathbb{E}\left[\xi_{j+1} \ind{B_{j}} | \mathcal{F}_{j} \right]$ and remark that $\mathbb{E}\left[ \left\| \epsilon_{j+1} \right\|^{2} |\mathcal{F}_{j} \right] \leq 2C$. Then, recalling that $\delta_{n} = \gamma_{n}^{-1/2}(\ln n)^{-1/2}$, and since for all $k \geq 2$, 
\[
\mathbb{E}\left[ \left\| \epsilon_{j+1}\right\|^{k} |\mathcal{F}_{j} \right] \leq 2^{k-2}\delta_{j}^{k-2} \mathbb{E}\left[ \left\| \xi_{j+1} \right\|^{2} \ind{B_{j}} |\mathcal{F}_{j} \right] \leq 2^{k-1}C\delta_{j}^{k-2} ,
\]
\begin{align*}
e_{j,n}  = \sum_{k=2}^{\infty} \lambda^{k} \left\| \beta_{n,j}\right\|_{op}^{k} \gamma_{j}^{k} \mathbb{E}\left[ \left\|  \Xi_{j+1} \right\|^{k} |\mathcal{F}_{j} \right]  
& \leq \sum_{k=2}^{\infty} \lambda^{k} \left\| \beta_{n,j}\right\|_{op}^{k} \gamma_{j}^{k} v_{j}^{k} \mathbb{E}\left[ \left\| \epsilon_{j+1} \right\|^{k} |\mathcal{F}_{k} \right] \\
& \leq \sum_{k=2}^{\infty} \lambda^{k} \left\| \beta_{n,j}\right\|_{op}^{k} \gamma_{j}^{k} v_{j}^{k} 2^{k-1}C\delta_{j}^{k-2} \\
& \leq 2C\lambda^{2} \left\| \beta_{n,j} \right\|_{op}^{2} \gamma_{j}^{2}v_{j}^{2} \sum_{k=2}^{\infty} (2\lambda)^{k-2} \left\| \beta_{n,j}\right\|_{op}^{k-2}\gamma_{j}^{\frac{k-2}{2}} v_{j}^{k-2}\ln j^{-\frac{k-2}{2}} \\
& = 2C\lambda^{2} \left\| \beta_{n,j} \right\|_{op}^{2} \gamma_{j}^{2}v_{j}^{2} \exp \left( 2 \lambda \left\| \beta_{n,j}\right\|_{op}\sqrt{\gamma_{j}} v_{j} \right)
\end{align*}
Then,
\[
\mathbb{P}\left[ \left\| M_{2,n+1} \right\| \geq r \right] \leq \mathbb{P}\left[ 2G_{n+1} \geq \frac{e^{\lambda r}}{\prod_{j=1}^{n}\left( 1+  2C\lambda^{2} \left\| \beta_{n,j} \right\|_{op}^{2} \gamma_{j}^{2} v_{j}^{2}\exp \left( 2 \lambda \left\| \beta_{n,j}\right\|_{op} v_{j}\sqrt{\gamma_{j}\ln j} \right) \right) } \right]
\]
Applying Markov's inequality, 
\[
\mathbb{P}\left[ \left\| M_{2,n+1} \right\| \geq r \right] \leq 2 \exp \left( - \lambda r + 2C\lambda^{2} \sum_{j=1}^{n} \left\| \beta_{n,j} \right\|_{op}^{2} \gamma_{j}^{2}v_{j}^{2} \exp \left( 2 \lambda \left\| \beta_{n,j}\right\|_{op}v_{j}\sqrt{\gamma_{j}\ln j} \right) \right) .
\]
Take $\lambda = \gamma_{n}^{-1/2}v_{n}^{-1}\sqrt{\ln n}$. Let $C_{0} = \left\| \beta_{n_{0},0} \right\|_{op}$ and remark that for $n \geq 2n_{0}$ (i.e such that $\gamma_{n/2} \lambda_{\max}(\Gamma) \leq 1$),  and for all $j \leq n/2$,
\[
\left\| \beta_{n,j} \right\|_{op} \leq C_{0}\exp \left( - c\lambda_{\min}(n/2)^{1-\alpha} \right),
\]
so that for all $j \leq n/2$,
\[
\lambda \left\|\beta_{n,j} \right\|_{op} \gamma_{j}v_{j} \leq C_{0}\exp \left( - \lambda_{\min}(n/2)^{1-\alpha} \right)\sqrt{n^{2b+\alpha}\ln n} \limite{n\to + \infty}{a.s} 0 .
\]
Furthermore, for all $n \geq 2n_{0}$, and for all $j \geq n/2$, 
\[
\lambda \left\|\beta_{n,j} \right\|_{op} \frac{\sqrt{\gamma_{j}}v_{j}}{\sqrt{\ln j}} \leq C_{0} 2^{2b+\alpha+1}.
\]
Then, there is a positive constant $C''$ such that for all $n \geq 1$ and $j \leq n$,
\[
\exp \left( \lambda \left\|\beta_{n,j} \right\|_{op} \sqrt{\gamma_{j}}v_{j} \right) \leq C''
\]
Finally, one can easily check that (see Lemma E.2 in \cite{CG2015})
\[
\sum_{j=1}^{n} \left\| \beta_{n,j}\right\|_{op}^{2} \gamma_{j}^{2} \frac{(\ln j)^{2a}}{j^{2b}} =\mathcal{O} \left( \frac{(\ln n)^{2a}}{n^{2b+\alpha}} \right) .
\]
There is a positive constant $C'''$ such that
\[
\mathbb{P}\left[ \left\| M_{2,n+1} \right\| \geq r \right] \leq \exp \left( - r v_{n}^{-1}\gamma_{n}^{-1/2}\sqrt{\ln n} + C''' \ln n \right)
\]
Then , taking $r = \left( 2+C''' \right) v_{n} \sqrt{\gamma_{n}\ln n}$, it comes
\[
\mathbb{P}\left[ \left\| M_{2,n+1} \right\| \geq \left( 2+C''' \right) v_{n} \sqrt{\gamma_{n}\ln n} \right] \leq \exp \left( - 2 \ln n \right) = \frac{1}{n^{2}}
\]
and applying Borell Cantelli's lemma,
\[
\left\| M_{2,n+1} \right\| = \mathcal{O} \left( v_{n}\sqrt{\gamma_{n}\ln n}  \right) \quad a.s.
\]

\noindent\textbf{Bounding $M_{3,n+1} := \sum_{k=1}^{n} \beta_{n,k} \gamma_{k}  R_{n} \left( \xi_{k+1}\ind{C_{k+1}} - \mathbb{E}\left[ \xi_{k+1} \ind{C_{k+1}} |\mathcal{F}_{k} \right] \right)$. } Let us denote $\epsilon_{k+1} = \xi_{k+1}\ind{C_{k+1}} - \mathbb{E}\left[ \xi_{k+1} \ind{C_{k+1}} |\mathcal{F}_{k} \right]$ and remark that for $n \geq n_{0} $,
\begin{align*}
\mathbb{E}\left[ \left\| M_{3,n+1} \right\|^{2} |\mathcal{F}_{n} \right] & \leq  \left( 1- \lambda_{\min} \gamma_{n} \right) \left\| M_{3,n} \right\|^{2} +  \gamma_{n}^{2} v_{n}^{2} \mathbb{E}\left[ \left\| \epsilon_{n+1} \right\|^{2} |\mathcal{F}_{n} \right] \\
& \leq \left( 1- \lambda_{\min} \gamma_{n} \right) \left\| M_{3,n} \right\|^{2} +  \gamma_{n}^{2}v_{n}^{2} \mathbb{E}\left[ \left\| \xi_{n+1} \right\|^{2} \ind{\left\| \xi_{n+1} \right\|^{2} \geq \gamma_{n}^{-1}} |\mathcal{F}_{n} \right]
\end{align*}
Let $V_{n}' = \frac{1}{\gamma_{n}v_{n}^{2}} \left\| M_{3,n} \right\|^{2}$. There are a rank $n_{1}$ and a positive constant $c$ such that for all $n \geq n_{1}$
\[
\mathbb{E}\left[ V_{n+1} |\mathcal{F}_{n} \right] \leq \left( 1- c\gamma_{n} \right) V_{n} + \mathcal{O} \left( \gamma_{n}\mathbb{E}\left[ \left\| \xi_{n+1} \right\|^{2} \ind{\left\| \xi_{n+1} \right\|^{2} \geq \gamma_{n}^{-1}} |\mathcal{F}_{n} \right] \right)  \quad a.s.
\]
Applying Robbins-Siegmund Theorem as well as equation \eqref{hypmart}, it comes
\[
\left\| M_{3,n+1} \right\|^{2} = \mathcal{O} \left( \gamma_{n}v_{n}^{2} \right) \quad a.s.
\]

\end{proof}

\subsection{A useful theorem for averaged stochastic algorithms}
\begin{theo}\label{theomartmoy}
Let $H$ be a separable Hilbert space and let
\[
M_{n} = \frac{1}{n}\sum_{k=1}^{n} R_{k} \xi_{k+1},
\]
where 
\begin{itemize}
\item $\left( \xi_{n} \right)$ is a $H$-valued martingale differences sequence adapted to a filtration $\left( \mathcal{F}_{n} \right)$ verifying
\[
\mathbb{E}\left[ \left\| \xi_{n+1} \right\|^{2} |\mathcal{F}_{n} \right] \leq C + R_{2,n}
\]
where $(R_{2,n})_n$ converges almost surely to $0$.
\item $(R_{n})$ is a sequence of operators on $H$ such that for a deterministic sequence $\left( v_{n} \right)$,
\[
\left\| R_{n} \right\|_{op} = \mathcal{O} \left( v_{n} \right)\quad a.s  \quad \text{and} \quad \exists c \leq 1,(a_{n}), \quad \frac{v_{n}}{v_{n+1}} = 1 +  \frac{c}{n} +  \frac{a_{n}}{n} + o \left( \frac{a_{n}}{n} \right), 
\]
with $(a_{n})$ converging to $0$.
\end{itemize}
Then, for all $\delta > 0$,
\begin{itemize}
\item If $\sum_{n\geq 1} v_{n}^{2} < + \infty \quad a.s$, 
\[
\left\| M_{n}^{2} \right\|^{2} = \mathcal{O} \left( \frac{1}{n^{2}} \right) \quad a.s.
\]
\item If $c < 1/2$, 
\[
\left\| M_{n} \right\|^{2} = o \left( n^{-1}v_{n}^{-2}(\ln n)^{1+\delta} \right) \quad a.s.
\]
\item If $\sum_{n\geq 1}\frac{a_{n}}{n} < + \infty$ and if $1/2 \leq c \leq 1$, 
\[
\left\| M_{n} \right\|^{2}= o \left( n^{2c-2}v_{n}^{-2}(\ln n)^{1+\delta} \right) \quad a.s
\]
\item If $\sum_{n\geq 1}\frac{a_{n}}{n} = + \infty$ and if $1/2 \leq c < 1$, for all $a<2-2c$ 
\[
\left\| M_{n} \right\|^{2} = o \left( n^{-a}v_{n}^{-2} \right) \quad a.s.
\]

\end{itemize}

\end{theo}

\begin{proof}
If $\sum_{n\geq 1} v_{n}^{2} < + \infty$, let us consider $W_{n} = n^{2} \left\| M_{n}\right\|^{2}$. We have
\[
\mathbb{E}\left[ W_{n+1} |\mathcal{F}_{n} \right] \leq W_{n} + \left\| R_{n} \right\|_{op}  \mathbb{E}\left[ \left\| \xi_{n+1} \right\|^{2} |\mathcal{F}_{n} \right] .
\]
Then,
\[
\mathbb{E}\left[ W_{n+1} |\mathcal{F}_{n} \right] \leq W_{n} + \mathcal{O} \left( v_{n+1}^{2} \right) \quad a.s
\]
and applying Robbins-Siegmund Theorem,
\[
\left\| M_{n} \right\|^{2} = \mathcal{O} \left( \frac{1}{n^{2}} \right) .
\]
Let us consider $a \leq 1$ and $V_{n+1,a} = \frac{(n+1)^{a}}{v_{n+1}^{2}(\ln (n+1))^{1+\delta}} \left\| M_{n+1} \right\|^{2}$. Then,
\begin{align*}
\mathbb{E}\left[ V_{n+1,a} |\mathcal{F}_{n} \right] & =  \frac{(n+1)^{a}}{(\ln (n+1))^{1+\delta}} \left( \frac{n}{n+1} \right)^{2}\frac{v_{n}^{2}}{v_{n+1}^{2}}\left\| M_{n} \right\|^{2} + \frac{(n+1)^{a}}{(\ln (n+1))^{1+\delta}}\frac{1}{(n+1)^{2}} \frac{\left\| R_{n} \right\|_{op}}{v_{n+1}^{2}} \mathbb{E}\left[ \left\| \xi_{n+1} \right\|^{2} |\mathcal{F}_{n} \right] \\
& \leq \left( \frac{n}{n+1}\right)^{2-a}\frac{v_{n+1}^{2}}{v_{n}^{2}} 	 V_{n} + \mathcal{O} \left( \frac{1}{n^{2-a}(\ln n)^{1+\delta}} \right) \quad a.s \\
& = \left( 1- (2-a -2c) \frac{1}{n} + \frac{a_{n}}{n} + \mathcal{O} \left( \frac{1}{n^{2}} \right) \right) V_{n} + \mathcal{O} \left( \frac{1}{n^{2-a}(\ln n)^{1+\delta}} \right) \quad a.s
\end{align*}
Applying Robbins-Siegmund theorem, 
\begin{itemize}
\item If $c < 1/2$, one can take $a= 1$ and
\[
\left\| M_{n} \right\|^{2} = o \left( n^{-1}v_{n}^{2}(\ln n)^{1+\delta} \right) \quad a.s.
\]
\item If $\sum_{n\geq 1}\frac{a_{n}}{n} < + \infty$ and if $1/2 \leq c \leq 1$, one can take $a=2-2c$ and
\[
\left\| M_{n} \right\|^{2}= o \left( n^{2c-2}v_{n}^{2}(\ln n)^{1+\delta} \right) \quad a.s
\]
\item If $\sum_{n\geq 1}\frac{a_{n}}{n} = + \infty$ and if $1/2 \leq c < 1$, for all $a<2-2c$ 
\[
\left\| M_{n} \right\|^{2} = o \left( n^{-a}v_{n}^{2}(\ln n)^{1+\delta} \right) \quad a.s.
\]

\end{itemize}

\end{proof}

\bibliographystyle{apalike}
\bibliography{biblio_redaction_2}

\end{document}